\documentclass[12pt]{amsart}
\usepackage{amsmath,amssymb,amsthm,fullpage,amscd,mathrsfs}
\usepackage[titletoc,title]{appendix}

\newtheorem{theorem}{Theorem}[section]
\newtheorem{proposition}[theorem]{Proposition}
\newtheorem{lemma}[theorem]{Lemma}
\newtheorem{example}[theorem]{Example}
\newtheorem{corollary}[theorem]{Corollary}
\theoremstyle{definition}

\newtheorem{remark}[theorem]{Remark}
\newtheorem*{assumptions}{Assumptions}

\newcommand\GL{\operatorname{GL}}
\newcommand\SL{\operatorname{SL}}
\newcommand\Mat{\operatorname{Mat}}
\newcommand\Ind{\operatorname{Ind}}
\newcommand\C{\mathbb{C}}
\newcommand\Z{\mathbb{Z}}
\newcommand\Treg{{\widehat{T}_{\operatorname{reg}}}}
\newcommand{\sage}{{\sc Sage}}
\newcommand{\Aa}{{\mathcal{A}}}
\newcommand{\Ai}{{A}}

\begin{document}

\title{Hecke Modules from Metaplectic Ice}
\author{Ben Brubaker}
\address{School of Mathematics, University of Minnesota, Minneapolis, MN 55455}
\email{brubaker@math.umn.edu}
\author{Valentin Buciumas}
\address{Einstein Institute of Mathematics, Edmond J. Safra Campus, Givat Ram, The Hebrew University of Jerusalem, Jerusalem, 91904, Israel}
\email{valentin.buciumas@gmail.com}
\author{Daniel Bump}
\address{Department of Mathematics, Stanford University, Stanford, CA 94305-2125}
\email{bump@math.stanford.edu}
\author{Solomon Friedberg}
\address{Department of Mathematics, Boston College, Chestnut Hill, MA 02467-3806}
\email{solomon.friedberg@bc.edu}

\begin{abstract}
We present a new framework for a broad class of affine Hecke algebra
modules, and show that such modules arise in a number of settings
involving representations of $p$-adic groups and $R$-matrices for
quantum groups. Instances of such modules arise from (possibly
non-unique) functionals on $p$-adic groups and their metaplectic covers,
such as the Whittaker functionals. As a byproduct, we obtain new,
algebraic proofs of a number of results concerning metaplectic Whittaker
functions. These are thus expressed in terms of metaplectic versions of
Demazure operators, which are built out of $R$-matrices of quantum
groups depending on the cover degree and associated root system.
\end{abstract}
\maketitle


An important tool in studying representations of $p$-adic groups is the affine Hecke algebra.
Ion~\cite{IonMatrix}, Brubaker, Bump and Licata~{\cite{BrubakerBumpLicata}},
Brubaker, Bump and Friedberg~{\cite{BrubakerBumpFriedbergMatrix, BrubakerBumpFriedbergPacificJ}} and other papers considered
representations of Hecke algebras related to models of representations of
$p$-adic groups. These works mainly focused on unique functionals such as the
Whittaker, spherical and Bessel functionals, but, as we shall show here, the
ideas are also applicable to certain non-unique functionals on central
extensions of $p$-adic groups. The goal is to show that these functionals act
as Hecke algebra module maps to various target spaces. This results in
recursion formulas for Iwahori fixed vectors in the model using Demazure-like
operators based on the Hecke action. It allows for effective computation,
provides links to the geometry of associated Bott-Samelson varieties, and
proves that these functions match important classes of symmetric functions,
such as various specializations of non-symmetric Macdonald polynomials.

In a separate direction, Brubaker, Bump and Friedberg~\cite{hkice} showed that
the values of the spherical Whittaker function for unramified principal series
of $GL_r(F)$, where $F$ is a $p$-adic field, may be realized as partition
functions of a solvable lattice model. A similar model for the spherical
Whittaker functions for covers of $GL_r(F)$ -- a situation in which the
Whittaker functional is no longer unique -- was given in Brubaker, Bump,
Chinta, Friedberg and Gunnells~\cite{mice}.  However, except when the cover is
trivial, this model was not solvable in the sense of Baxter~\cite{Baxter}. The
Boltzmann weights were altered to obtain a solvable model in
Brubaker, Buciumas and Bump~{\cite{BrubakerBuciumasBump}}. This is the
``metaplectic ice" of the title.  Moreover, {\cite{BrubakerBuciumasBump}}
provided a new algebraic framework, relating the models to the $R$-matrices of
a Drinfeld twist of the quantum affine Lie superalgebra $U_q
(\widehat{\mathfrak{g}\mathfrak{l}} (n| 1))$. As a biproduct of this
discovery, a relationship between the scattering matrices of the intertwining
integrals on Whittaker coinvariants and $R$-matrices of a Drinfeld twist of
$U_q(\widehat{\mathfrak{g}\mathfrak{l}} (n))$ came to light.

In the present paper, we will unify these directions by providing simultaneous
foundations for both. The search for new foundations is motivated by the
following consideration. In \cite{BrubakerBumpFriedbergMatrix}, we took
the point of view of the \textit{universal principal series}. This
means that the induced character becomes $\mathbb{C}[\Lambda]$-valued, where
$\Lambda$ is the coweight lattice of the associated algebraic group $G$,
or equivalently, the weight lattice of the Langlands dual group $\widehat{G}$.
On the other hand, in \cite{BrubakerBuciumasBump}, this does not work
so well, since the the complex parameters of the induced character become
spectral parameters in the $R$-matrix. To unify these directions, we
take a new approach making use of a \textit{representation schema} -- an axiomatic
description of certain representations of the affine Hecke algebra.  We then
give examples of such schema coming from the various aforementioned sources:
from representations of $p$-adic groups or their covers, from models for
various functionals on these representations, and from $R$-matrices of quantum
groups.  We explain how these examples are related and use the formalism to
extend the work of \cite{BrubakerBumpFriedbergMatrix} to the metaplectic
setting, where metaplectic versions of Demazure operators as in
Chinta, Gunnells and Pusk\'as~\cite{ChintaGunnellsPuskas}, using the action of
Chinta and Gunnells~\cite{ChintaGunnells}, appear naturally and inherit an $R$-matrix
interpretation.

We now describe our results in more detail, first introducing the affine Hecke algebra. 
Let $W$ be a finite Weyl group, acting on the root system $\Phi^{\vee}$ inside
$\Lambda$. Thus $W$ is a finite Coxeter group with simple reflections
$s_i$. Let $n (i, j)$ be the order of $s_i s_j$ when $i\neq j$. We will denote
by $\alpha_i^{\vee}$ the simple roots.

The \textit{affine Hecke algebra} has the following presentation, due to
Bernstein. It has generators $T_i$ corresponding to the simple roots, together
with an abelian subalgebra $\theta_{\Lambda}$ isomorphic to the group
algebra of $\Lambda$. The $T_i$ satisfy the quadratic relations $T_i^2
= (v - 1) T_i + v$, where $v$ is an indeterminate or a complex number, and the
braid relations $T_i T_j T_i \cdots = T_j T_i T_j \cdots$, where there are $n
(i, j)$ factors on both sides. The algebra $\theta_{\Lambda}$ has a
basis $\theta_{\lambda}$ with $\lambda \in \Lambda$. The
\textit{Bernstein relation} is also required:
\begin{equation}
  \label{bernstein} \theta_{\lambda} T_i - T_i \theta_{s_i \lambda} = (v - 1) 
  \frac{\theta_{\lambda} - \theta_{s_i \lambda}}{1 - \theta_{-
  \alpha_i^{\vee}}} .
\end{equation}
Note that on the right-hand side the numerator is divisible by the denominator in
$\theta_{\Lambda}$. We will denote by
$\widetilde{\mathcal{H}}_v$ the algebra generated by the $T_i$ and
$\theta_{\Lambda}$ with these generators and relations. We will also
denote by $\mathcal{H}_v$ the finite-dimensional subalgebra generated by just
the $T_i$. Then $\dim \hspace{0.17em} \mathcal{H}_v = |W|$ and as a vector
space $\widetilde{\mathcal{H}}_v = \mathcal{H}_v \otimes
\theta_{\Lambda}$.

We will describe various representations of
$\widetilde{\mathcal{H}}_v$ axiomatically. 
One instance 
comes from the theory of $p$-adic groups, so let us start with that.
Suppose that $G$ is a split reductive group over a nonarchimedean local field
$F$ with residue cardinality $q$. Let $T$ be a split maximal torus, and assume
that $\Lambda$ is identified with the cocharacter group $X_{\ast} (T)$.
Then $\Lambda$ is also identified with the group $X^{\ast} (\widehat{T})$
of rational characters of the dual torus $\widehat{T}$. This is the maximal torus
in the (connected) Langlands dual group $\widehat{G}$. We will then assume that
$\Phi^{\vee}$ is the root system of $\widehat{G}$. Let $J$ be an Iwahori subgroup
of $G (F)$. Then by combining results of Iwahori and Matsumoto~{\cite{IwahoriMatsumoto}} and
Proposition~3.6 of Lusztig~{\cite{LusztigGraded}}, $\widetilde{\mathcal{H}}_q$ is
isomorphic to the convolution algebra $\mathcal{H}_J$ of compactly-supported
$J$-bi-invariant functions on $G (F)$. Therefore if $(\pi, V)$ is any $G (F)$
representation, then $\widetilde{\mathcal{H}}_q$ acts on the space $V^J$ of
$J$-fixed vectors.

On the other hand, there is another way that the affine Hecke algebra enters
the representation theory of $G (F)$, namely through the intertwining
operators. If $\mathbf{z} \in \widehat{T} (\mathbb{C})$, then $\mathbf{z}$
parametrizes a principal series representation $M (\mathbf{z})$ of $G (F)$,
and if $w$ is any element of the Weyl group $W$, there is an intertwining
operator, defined by an integral, $\Ai_w : M (\mathbf{z}) \to M (w \mathbf{z})$.
Such operators play an important role in many aspects of the representation theory.
Accordingly, it is key that one may use the Hecke algebra to model the
intertwining operators $\Ai_w$. There are at least three different approaches.

First, there is the method of Rogawski~{\cite{Rogawski}}, which begins with the
observation that the space of $J$-fixed vectors in $M (\mathbf{z})$ can be
identified with $\mathcal{H}_v$. (Both are $|W|$-dimensional vector spaces.)
Thus the intertwining operators $\Ai_{s_i}$ become endomorphisms of
$\mathcal{H}_q$, and it is possible to give formulas for them in terms of
$T_i$, depending on $\mathbf{z}$.

The second method replaces the representations $M (\mathbf{z})$ with the
\textit{universal principal series}, which are representations induced not
from the Borel subgroup but from its maximal unipotent subgroup extended by $T
(\mathfrak{o})$, where $\mathfrak{o}$ is the ring of integers of the local
field $F$. For this, see Chriss and Khuri-Makdisi~{\cite{ChrissKhuriMakdisi}} and
Haines, Kottwitz and Prasad~{\cite{HainesKottwitzPrasad}}.

Here we will give a third approach. We will proceed axiomatically
from data consisting of vector spaces $\mathcal{M} (\mathbf{z})$ and maps
$\Aa_{s_i} : \mathcal{M} (\mathbf{z}) \longrightarrow \mathcal{M}
(s_i \mathbf{z})$ when $s_i$ is a
simple reflection. Assuming simple axioms, the \textit{representation schema}
which we describe in Section~\ref{schemasect} below
produces an action of $\tilde{\mathcal{H}}_v$ on
\[\mathfrak{M}(\mathbf{z})=\bigoplus_{w \in W}\mathcal{M} (w\mathbf{z}).\]
We emphasize that we give different instances of the schema in this paper and so
the meaning of $\mathcal{M}(\mathbf{z})$ and $\mathfrak{M}(\mathbf{z})$
is not the same throughout the text.

There are a number of different examples of the schema attached to a $p$-adic group
(without even passing to a covering group).
As a first example we may take $v = q^{- 1}$, where $q$ is the residue cardinality, and $\mathcal{M}
(\mathbf{z})$ to be the principal series representation $M (\mathbf{z})$.
Thus we get an action of the affine Hecke algebra
$\widetilde{\mathcal{H}}_{q^{- 1}}$ on
\begin{equation}
  \label{mtheory} \bigoplus_{w \in W} M (w \mathbf{z}) .
\end{equation}
These spaces are all infinite dimensional. The construction initially assumes
$\mathbf{z}$ is in the regular set $\widehat{T}_{\text{reg}}$, that is,
$\mathbf{z}^{\alpha^{\vee}}\neq 1$ for any positive root $\alpha^{\vee}$,
but it is possible to remove this assumption in certain cases.

Because the action is built from the intertwining operators, it commutes with
the action of $G (F)$. So if $K$ is any compact subgroup of $G (F)$, we may
take $K$-invariants, and we obtain a left action of
$\widetilde{\mathcal{H}}_{q^{- 1}} \otimes \mathcal{H}_K$ on
\[ \mathfrak{M}(\mathbf{z})^K := \bigoplus_{w \in W} M (w \mathbf{z})^K . \]
For example, if $K$ is the special maximal compact subgroup $K^{\circ}$, then
to each $w \in W$, the space $M (w \mathbf{z})^{K^{\circ}}$ is one-dimensional
and so $\mathfrak{M}^{K^\circ}$ is $|W|$-dimensional.  Additional structure beyond the action
of a single affine Hecke algebra will arise in other examples, including
those from $R$-matrices, as well.

If instead $K = J$, then $\mathfrak{M}^J$ is $|W|^2$-dimensional and is a
module for two copies of the affine Hecke algebra. With this choice
the decomposition
\begin{equation}
  \label{mjh2decomp}
  \mathfrak{M}^J = \bigoplus_{w \in W} \mathcal{M}(w\mathbf{z}),
  \qquad \mathcal{M}(w\mathbf{z})=M (w \mathbf{z})^J
\end{equation}
of the 
$\widetilde{\mathcal{H}}_{q^{-1}} \otimes \widetilde{\mathcal{H}}_{q}$-module
is into $\mathfrak{M}^J=\widetilde{\mathcal{H}}_{q}$-invariant subspaces.
Here $\widetilde{\mathcal{H}}_{q^{-1}}$ refers to the affine
Hecke algebra acting through the representation schema
(Theorems~\ref{hecketheorem} and~\ref{hecketheoremaffine}),
whereas $\mathcal{H}_J\cong\widetilde{\mathcal{H}}_q$ refers to the action
on each summand $M(w\mathbf{z})^J$ by convolution. (The two
Hecke algebras with parameters $q^{-1}$ and $q$
are actually isomorphic, but we will not make use of this.)

By symmetry we may hope for another decomposition
\begin{equation}
  \label{mjh1decomp}
  \mathfrak{M}^J = \bigoplus_{w \in W} \mathcal{M}'(w \mathbf{z})
\end{equation}
into $|W|$-dimensional $\widetilde{\mathcal{H}}_{q^{-1}}$-invariant
subspaces. This may be accomplished by decomposing with
respect to the abelian subalgebra $\theta_{\Lambda}$ of
$\mathcal{H}_J=\widetilde{\mathcal{H}}_{q}$. It may be
shown that $|W|$ different characters of $\theta_{\Lambda}$
occur, and decomposing with respect to these gives the second
decomposition (\ref{mjh1decomp}).

Instead of taking $\mathcal{M}(\mathbf{z})$ to be the invariants
of $M(\mathbf{z})$ with respect to some subgroup, we may take
the module of coinvariants. Specifically, let $U$ be a subgroup of $G (F)$,
and let $\psi$ be a character of $U$. The 
\textit{module of coinvariants} of a $G (F)$-module $(\pi, V)$ is
\[ V_{U, \psi} = V / \langle \pi (u) v - \psi (u) v \rangle . \]
We may take $\mathcal{M}(\mathbf{z})$ to be $M(\mathbf{z})_{U,\psi}$
in the schema.  Then 
\begin{equation}
  \label{whitcoimod} \bigoplus_{w \in W} M (w \mathbf{z})_{U, \psi}
\end{equation}
becomes a module for $\widetilde{\mathcal{H}}_{q^{- 1}}$.

For example, let $U$ be the maximal unipotent subgroup
generated by the positive roots. Its normalizer is the positive
Borel subgroup $B$, and with $\psi$ the trivial character,
$M(\mathbf{z})_{U,1}$ is the usual $|W|$-dimensional
Jacquet module $\operatorname{Jac}(M(\mathbf{z}))$, which is a module for $T(F)$. With
$\mathbf{z}\in\Treg(\C)$, the Jacquet module is a direct sum of $|W|$
distinct quasicharacters $\chi_{w\mathbf{z}}$ of $T(F)$.
Since the action of $T(F)$ commutes with the action of
$\widetilde{\mathcal{H}}_{q^{-1}}$, we may then decompose
this 
\[\bigoplus_{w\in W}\operatorname{Jac}(M(\mathbf{z}))=
\bigoplus_{w\in W} \mathfrak{M}_{\chi_{w\mathbf{z}}\text{-isotypic}},\]
where $\mathfrak{M}_{\chi_{w\mathbf{z}}\text{-isotypic}}$ is the
$\chi_{w\mathbf{z}}$ isotypic part. On both sides of this
equation, we have $|W|$ terms, each a $|W|$-dimensional
vector space. However on the left-side, the summands are
\textit{not} $\widetilde{\mathcal{H}}_{q^{-1}}$ invariant.
On the right side, they are. This decomposition is essentially
the same as the decomposition (\ref{mjh1decomp}). Indeed, it
was shown by Casselman~\cite{CasselmanSpherical}, Proposition~2.4,
that the projection of the Iwahori fixed vectors onto the
Jacquet module is an isomorphism, making it possible to compare
the two decompositions.

To give another example involving coinvariants, 
let $U$ be the maximal unipotent subgroup as before,
but now let $\psi$ be a
nondegenerate character of $U$. Then with $V = M (\mathbf{z})$, the
twisted Jacquet module $V_{U, \psi}$ is the module of coinvariants for the
Whittaker model, which is one-dimensional. Thus we obtain another
$|W|$-dimensional module. We may allow $\mathbf{z}$ to vary in
(\ref{whitcoimod}), and regard this $|W|$-dimensional vector space, depending
on a parameter $\mathbf{z}$, as a vector bundle over $\widehat{T}
(\mathbb{C})$. This comes with an action of the Weyl group, and we can ask for
the equivariant sections. This is another module which may be identified with
the space $\mathcal{O} (\widehat{T}(\mathbb{C}))$ of regular functions on
$\widehat{T}(\mathbb{C})$.  In this representation the $T_i$ act by
Demazure-like operators.

Of note, we can give an interpretation of the formula of
Casselman and Shalika~\cite{CasselmanShalika} in this context that is different from the usual
one. That is, we may take the spherical idempotent in the Hecke algebra and
apply it in the module $\mathcal{O} (\widehat{T} (\mathbb{C}))$. The
Casselman-Shalika formula emerges (see Theorem~\ref{abstractcs}). This point of view may
be used to give a new approach to
the investigation of {\cite{BrubakerBumpFriedbergMatrix}} relating functionals
to Hecke algebra representations.

Below, we will generalize this by replacing $G (F)$ by any
$n$-fold metaplectic cover $\widetilde{G}$ of $G (F)$ provided that $F$ contains a copy of the
$n$-th roots of unity. As we will now explain, the Hecke modules we obtain
in these metaplectic cases are for a different Hecke algebra. To explain this
point we must discuss the \textit{metaplectic L-group}, a generalization of
the usual (connected) Langlands L-group which arises in the
case $n=1$. The metaplectic L-group plays a role in understanding metaplectic
correspondences such as the Shimura correspondence as functorial
lifts. (See {\cite{BumpFriedbergGinzburgLifting, McNamaraPrincipal,Weissman, GanGao}.)
For us, it will give a convenient language for discussing
Hecke algebra representations that also occur in the context of quantum groups.

Motivated by the Shimura correspondence, Savin~\cite{SavinShimura} computed
the genuine Iwahori Hecke algebra of $\widetilde{G}$. He showed that this is
the same as the Hecke algebra of the $F$-points of another reductive group
${G}'$ (not a metaplectic cover).  For example, suppose that $G$ is
``doubly-laced,'' i.e.\ its Dynkin diagram contains a double bond. Then if $n$
is even, the long and short roots of $G$ and $G'$ are reversed, while if $n$
is odd, they are not. We will observe a similar phenomenon for the Hecke
algebra representations that we construct.

Now $G'$ itself has an L-group, whose connected component of the identity we denote $\widehat{G}'(\C)$.
The group $\widehat{G}'(\C)$ can be regarded as the connected L-group of the metaplectic
group $\widetilde{G}$, and the intermediary $G'$ may be dispensed with. This
\text{metaplectic L-group} $\widehat{G}'(\C)$ is mentioned (without giving a
formal definition) in Bump, Friedberg and
Ginzburg~\cite{BumpFriedbergGinzburgLifting}, where it was used heuristically
to predict metaplectic examples of functoriality.  A formal definition requires describing
the root datum for the reductive group $\widehat{G}'(\C)$, and this was done
by McNamara~\cite{McNamaraPrincipal} and Weissman~\cite{Weissman}. (See
\cite{GanGao} for a further useful account extending~\cite{Weissman}.) Weissman's
more sophisticated treatment is more general since he requires fewer roots of
unity in the ground field, and he also constructs the Galois part of the
L-group, not just the connected component. In this paper, McNamara's
description is sufficient for us, and we will use \cite{McNamaraPrincipal} as
our basic reference for the metaplectic L-group.

Returning to the Hecke algebra, let $G$ be a split reductive group over a $p$-adic
field, and $\widehat G(\C)$ be its connected Langlands dual group.  In view of the
historical origins of Hecke algebras as convolution rings of functions on
$G(F)$, we are accustomed to think of the Hecke algebras as associated with
$G(F)$, but it may be better to think of them as associated with
$\widehat{G}(\C)$. For example the spherical Hecke algebra is interpreted in
the Satake isomorphism as a ring of functions on $\widehat{G}(\C)$, and this
interpretation extends to the Iwahori Hecke algebra where (via the Bernstein
presentation) it is the weight lattice of $\widehat{G}(\C)$, not $G$, that
controls the structure of the Hecke algebra.  So we regard the Hecke algebra
as associated with $\widehat{G}(\C)$.

With this in mind, the representations that we construct from metaplectic
groups are representations of the affine Hecke algebra associated with the
metaplectic dual group $\widehat{G}'(\C)$, not with $\widehat{G}(\C)$. (See
Example~\ref{metalmtheory}.)  This is similar to the phenomenon for the
Iwahori Hecke algebra in \cite{SavinShimura} mentioned above.

So far we have discussed Hecke modules arising from Whittaker models on
$p$-adic groups or their covers as a source of examples of the axiom schema producing
Hecke modules. Alternatively, we will produce examples of the representation schema by
taking $\mathcal{M} (\mathbf{z})$ to be modules of quantum groups. For example
let us take $\widehat{T} (\mathbb{C})$ to be $(\mathbb{C}^{\times})^r$. If $z
\in \mathbb{C}^{\times}$ then the affine quantum group $U_{\sqrt{v}}
(\widehat{\mathfrak{gl}} (n))$ has an $n$-dimensional {\textit{evaluation
module}} $V_z$ indexed by $z$.\footnote{In {\cite{BrubakerBuciumasBump}} we
wrote of $U_v (\widehat{\mathfrak{gl}} (n))$ instead of $U_{\sqrt{v}}
(\widehat{\mathfrak{gl}} (n))$, but this change is just a difference of
conventions.} With $\mathbf{z}= \operatorname{diag} (z_1,
\cdots, z_r) \in \widehat{T} (\mathbb{C})$, then we may take $\mathcal{M}
(\mathbf{z})$ to be $V_{z_1} \otimes \cdots \otimes V_{z_r}$. The operators
$\Aa_{s_i} : \mathcal{M} (\mathbf{z}) \longrightarrow \mathcal{M} (s_i \mathbf{z})$ then
come from the {\textit{$R$-matrices}} that are intertwining operators $V_{z_i}
\otimes V_{z_{i + 1}} \longrightarrow V_{z_{i + 1}} \otimes V_{z_i}$,
reflecting the fact that $U_{\sqrt{v}} (\widehat{\mathfrak{gl}} (n))$-modules
form (almost) a braided category.\footnote{We say \textit{almost} a braided category
since the axioms of Joyal and Street~\cite{JoyalStreet} are almost satisfied. The exception
is that in rare cases the commutativity constraint $U\otimes V\to
V\otimes U$ may fail to be an isomorphism.}
Again, we obtain an $\tilde{\mathcal{H}}_v$-module.

The module $\mathfrak{M}(\mathbf{z})=\bigoplus_{w \in W}\mathcal{M} (w\mathbf{z})$ is isomorphic as a vector space to a direct sum of $r!$ copies of $\otimes^r \mathbb{C}^n$.
Because it is built from $U_{\sqrt{v}}({\widehat{\mathfrak{gl}}}(n))$-modules
using intertwining operators, we have commuting actions of
$\widetilde{\mathcal{H}}_v$ and $U_{\sqrt{v}}({\widehat{\mathfrak{gl}}}(n))$.
A related though simpler situation was found by Jimbo~{\cite{JimboHecke}} as a deformation of
Schur-Weyl duality. This gives commuting actions of $\mathcal{H}_v$ and $U_{\sqrt{v}}
(\mathfrak{gl} (n))$ (not affine) on one copy of $\otimes^r \mathbb{C}^n$.
We will explain a relationship between the representation of
$\widetilde{\mathcal{H}}_v \otimes U_{\sqrt{v}} (\widehat{\mathfrak{gl}} (n))$ that we
construct, and the representation of $\mathcal{H}_v \otimes U_{\sqrt{v}} (\mathfrak{gl}
(n))$ of Schur-Weyl-Jimbo duality.

Another representation of the Hecke algebra that can be
derived from the representation schema may be found in
Ginzburg, Reshetikhin and Vasserot~\cite{GinzburgReshetikhinVasserot}
and Kashiwara, Miwa and Stern~\cite{KashiwaraMiwaStern}.
This commutes with a Fock space representation of
$U_{\sqrt{v}}(\widehat{\mathfrak{gl}}(n))$ that is related
to the partition function interpretation of metaplectic
Whittaker functions in~\cite{BrubakerBuciumasBump}. We hope
to return to this point in a later paper.

We thus have two seemingly distinct sources of Hecke modules, one coming
from the representation theory of metaplectic covers of $p$-adic groups,
the other coming from modules of $U_{\sqrt{v}}(\widehat{\mathfrak{gl}} (n))$.
The key point for us is that these different modules are related. To accomplish
this, we must modify the quantum group by \textit{Drinfeld twisting}.
This is a modification of the Hopf algebra that changes both the
comultiplication and the $R$-matrices, but still allows one to define a Hecke
algebra representation with the properties mentioned above. The idea is to
introduce Gauss sums into the $R$-matrix to make it match the complicated
scattering matrix of the intertwining
operators on the Whittaker coinvariants for the metaplectic unramified
principal series representations for the $n$-fold metaplectic cover of $\GL(r)$,
which was computed by Kazhdan and Patterson~\cite{KazhdanPatterson}.
Thus, using Theorem~1 in {\cite{BrubakerBuciumasBump}} we will prove that
the metaplectic $V_{U, \psi}$ produces the same $\tilde{\mathcal{H}}_v$-module
as a twist of the $U_{\sqrt{v}} (\widehat{\mathfrak{gl}} (n))$ example. This
at last is the connection between Hecke modules coming from intertwining
operators of representations of $p$-adic groups and others coming from
$R$-matrices of quantum groups. We emphasize that $n$ is the degree of the
cover, \textit{not} the rank of the corresponding Lie group.

The last section of the paper applies our schema formalism to covers of any
split, reductive group. There we construct modules for the Hecke algebra of
the corresponding metaplectic dual group $\widehat{G}'(\C)$ as described above.
We use it to give a uniform approach to
a large class of (not necessarily unique) models. In an extended example at
the end of the section, we discuss the metaplectic Whittaker functional and
show how metaplectic Demazure-Whittaker operators (explored previously in
\cite{ChintaGunnellsPuskas,PuskasWhittaker,PatnaikPuskas})
arise naturally from our framework.

This paper is organized as follows.  Section~\ref{schemasect} introduces the representation
schema whose structure unifies these different objects.  We show that a small set of axioms
leads to a natural representation of the affine Hecke algebra (Theorem~\ref{hecketheoremaffine}).
In Section~\ref{whitheck} the schema is developed in the case of Whittaker models for $p$-adic groups,
and the use of the Hecke algebra is illustrated by a proof of the Casselman-Shalika formula as described above.
Section~\ref{HeckeRM} explains how one may use the schema in order
to create Hecke modules out of $R$-matrices.  By taking a certain limit, we also
obtain a module of the finite Hecke algebra which is a wreathed version of the
module used in Schur-Weyl-Jimbo duality; see Theorem~\ref{wreaththm}.
Section~\ref{Drinfeldtwist} reviews some facts about Drinfeld twisting for
quantum groups and $R$-matrices and explains how the construction in 
Section~\ref{HeckeRM} works if we do a Drinfeld twist on the quantum group. 
In Section~\ref{whitandr} we treat the
case of covering groups, obtaining modules for the affine Hecke algebra of the
metaplectic L-group from both representations of metaplectic groups and (in
the case of covers of $\GL(r)$) from $R$-matrices of quantum groups.
Then, relying on \cite{BrubakerBuciumasBump}, we show how these
representations coming from different sources -- metaplectic groups on
the one hand and quantum groups on the other -- can be identified (Remark~5.6).  The last section,
Section~\ref{modelsohboy}, applies our schema in the context of models for
representations of metaplectic covers of $p$-adic groups.  The goal is to
compute the action of the associated space of functionals on translates of the
spherical vector.  For a cover of a general split reductive $p$-adic group we
show how this is related to the schema of Section~\ref{schemasect} and how the
schema may be used to calculate the desired action
(Proposition~\ref{calc-CS}).  Then we show in Theorem~\ref{met-dz} that the
metaplectic Demazure operator of \cite{ChintaGunnellsPuskas} is obtained from
the affine Hecke action.

\medbreak\textit{Acknowledgments}.
\thanks{This work was supported by NSF grants DMS-1406238 (Brubaker), 
  DMS-1601026 (Bump), and DMS-1500977 (Friedberg) and by the 
Max Planck Institute for Mathematics (Buciumas).
We would like to thank Sergey Lysenko and Anna Pusk\'as for useful conversations,
and the referee for helpful comments.}

\section{A Representation Schema for the Hecke Algebra\label{schemasect}}

We will define representations of $\widetilde{\mathcal{H}}_v$ axiomatically.
There are several instances of the axioms, coming from either quantum
groups or from representations of $p$-adic groups, which we will catalog in
future sections.

Let $\Lambda$ be a rank $r$ weight lattice, which we identify with the
group of rational characters of a complex torus $\widehat{T}$.
Thus $\Lambda=X^\ast(\widehat{T})$. We require inside of $\Lambda$
a root system, which we will denote $\Phi^\vee$. As usual, we partition
this into positive and negative roots, and then the Weyl group $W$ is
generated by simple reflections $s_i$ which correspond to simple roots
$\alpha_i^\vee$.

Define $\Treg$ to be the regular subset of $\widehat{T}$. Thus, $\Treg(\mathbb{C})$
is the open set of all $\mathbf{z} \in \widehat{T}(\mathbb{C})$ such
that $\mathbf{z}^{\alpha^\vee}\neq 1$ for all $\alpha^\vee \in
\Phi^{\vee}$. Note that if $\mathbf{z}\in\Treg(\C)$ then $w\mathbf{z}\in\Treg(\C)$ for each
$w\in W$. For each such $\mathbf{z} \in \Treg(\C)$, we assume there to be given a complex 
vector space $\mathcal{M}(\mathbf{z})$ with some additional data that will
give rise to a representation of $\widetilde{\mathcal{H}}_v$ on
\[ \mathfrak{M} (\mathbf{z}) := \bigoplus_{w \in W} \mathcal{M} (w\mathbf{z}) . \]

The following assumptions are in force for Theorems~\ref{hecketheorem}
and~\ref{hecketheoremaffine} below.

\begin{assumptions} \label{assofuandme}
For each simple reflection $s_i$ in $W$, we assume there exists a linear map
$\Aa_{s_i}^{\mathbf{z}} : \mathcal{M} (\mathbf{z}) \longrightarrow \mathcal{M} (s_i
\mathbf{z})$ such that
$\Aa^{s_i \mathbf{z}}_{s_i} \circ \Aa_{s_i}^{\mathbf{z}}$ is a scalar linear
transformation of $\mathcal{M}(\mathbf{z})$. More precisely, we require
$\Aa^{s_i \mathbf{z}}_{s_i} \circ \Aa_{s_i}^{\mathbf{z}}$
to act as the scalar
\begin{equation}
  \label{doublesch}
  \left( \frac{1 - v\mathbf{z}^{\alpha_i^\vee}}{1 -\mathbf{z}^{\alpha_i^\vee}}
   \right) \left( \frac{1 - v\mathbf{z}^{- \alpha_i^\vee}}{1 -\mathbf{z}^{- \alpha_i^\vee}}
   \right) .
\end{equation}
Furthermore, we assume that
\begin{equation}
  \label{abraidc}
  \cdots \Aa_{s_i}^{s_j \mathbf{z}} \Aa_{s_j}^{\mathbf{z}} = \cdots \Aa_{s_j}^{s_i
  \mathbf{z}} \Aa_{s_i}^{\mathbf{z}}
\end{equation}
where the number of terms on each side is the order of $s_is_j$.
Thus if $s_i$ and $s_j$ commute, this means
\begin{equation}
  \label{abraidb}
\Aa_{s_i}^{s_j \mathbf{z}} \Aa_{s_j}^{\mathbf{z}} = \Aa_{s_j}^{s_i
  \mathbf{z}} \Aa_{s_i}^{\mathbf{z}}\,,
\end{equation}
while if $s_i s_j$ has order 3, then
\begin{equation}
  \label{abraida}
  \Aa_{s_i}^{s_j s_i \mathbf{z}} \Aa_{s_j}^{s_i \mathbf{z}}
  \Aa_{s_i}^{\mathbf{z}} = \Aa_{s_j}^{s_i s_j \mathbf{z}} \Aa_{s_i}^{s_j
  \mathbf{z}} \Aa_{s_j}^{\mathbf{z}},
\end{equation}
and so forth.
\end{assumptions}

By Matsumoto's theorem, (\ref{abraidb}) and (\ref{abraida}) imply that we may
define $\Aa_w^{\mathbf{z}} : \mathcal{M} (\mathbf{z}) \longrightarrow \mathcal{M} (w\mathbf{z})$ for
$w \in W$ such that $\Aa_{w w'}^{\mathbf{z}} = \Aa_w^{w' \mathbf{z}}
\Aa_{w'}^{\mathbf{z}}$ whenever $\ell (w w') = \ell (w) + \ell (w')$.
Here $\ell$ is the length function on $W$.

\bigbreak

We define an action of the Hecke algebra generators $\mathfrak{T}_i$ on $\mathfrak{M} (\mathbf{z})$ as
follows. Let $\phi = (\phi_{w\mathbf{z}}) \in \mathfrak{M} (\mathbf{z})$
where $\phi_{w\mathbf{z}} \in \mathcal{M} (w\mathbf{z})$ for each $w \in W$. Then we define

\begin{equation}
 \label{tiaction} 
  \mathfrak{T}_i \, : \, (\phi_{w \mathbf{z}}) \longmapsto (\psi_{w\mathbf{z}}) 
  \quad \text{where} \quad\psi_{w\mathbf{z}} =
  \left( \frac{(1 - v) (w\mathbf{z})^{
   \alpha_i^{\vee}}}{1 - (w\mathbf{z})^{\alpha_i^{^{\vee}}}} \right)
   \phi_{w\mathbf{z}} + \Aa_{s_i}^{s_i w\mathbf{z}} \phi_{s_i w\mathbf{z}}\,. 
\end{equation}

\begin{theorem}
  \label{hecketheorem}
  The operators $\mathfrak{T}_i$ satisfy the quadratic relation
  \begin{equation}
    \label{quadratic} \mathfrak{T}_i^2 = (v - 1) \mathfrak{T}_i + v
  \end{equation}
  and the braid relations
  \begin{equation}
    \label{firstbraid} \mathfrak{T}_i \mathfrak{T}_j \cdots
    =\mathfrak{T}_j \mathfrak{T}_i \cdots
  \end{equation}
  where the number of factors on both sides is the order of $s_is_j$.
  Thus (\ref{tiaction}) extends to an action of the finite Hecke algebra 
  $\mathcal{H}_v$ on $\mathfrak{M} (\mathbf{z})$.
\end{theorem}

\begin{proof}
  To check the quadratic relation, we note that the subspace $\mathcal{M}
  (w\mathbf{z}) \oplus \mathcal{M} (s_i w\mathbf{z})$ is invariant under
  $\mathfrak{T}_i$, and on this space $\mathfrak{T}_i + 1$ is represented by
  the block matrix
  \begin{equation}
    \label{tblockmat} \left(\begin{array}{cc}
      D_i  (w\mathbf{z}) & \Aa_{s_i}^{s_i w\mathbf{z}}\\
      \Aa_{s_i}^{w\mathbf{z}} & D_i  (s_i w\mathbf{z})
    \end{array}\right)
  \end{equation}
  where
  \[ D_i (w\mathbf{z}) = \left( \frac{(1 - v) (w\mathbf{z})^{\alpha_i^{\vee}}}
     {1 - (w\mathbf{z})^{\alpha_i^{^{\vee}}}} \right) =
     \frac{(v - 1)}{1 - (w\mathbf{z})^{- \alpha_i^{\vee}}} \,\,. \]
  We note that with $x = (w\mathbf{z})^{\alpha_i^{\vee}}$ we have
  \[ D_i (w\mathbf{z}) = \frac{(1 - v) x}{1 - x} \quad \text{and} \quad D_i (s_i
     w\mathbf{z}) = \frac{(1 - v) x^{- 1}}{1 - x^{- 1}} = \left( \frac{1 -
     v}{x - 1} \right) . \]
  Thus
  \begin{equation}
    \label{direl} D_i (w\mathbf{z}) + D_i (s_i w\mathbf{z}) = v - 1 .
  \end{equation}
  Using (\ref{direl}) and the fact that both $\Aa_{s_i}^{s_i w\mathbf{z}}
  \Aa_{s_i}^{w\mathbf{z}}$ and $\Aa_{s_i}^{w\mathbf{z}} \Aa_{s_i}^{s_i
  w\mathbf{z}}$ are represented by the same scalar ${(1 - v x) (1 - v
  x^{- 1})}{(1 - x)^{-1} (1 - x^{- 1})^{-1}}$, we see that on $\mathcal{M} (w\mathbf{z}) \oplus
  \mathcal{M} (s_i w\mathbf{z})$ the operator $\mathfrak{T}_i^2$ is
  represented by the matrix
  \[ \left(\begin{array}{cc}
       D_i  (w\mathbf{z}) & \Aa_{s_i}^{s_i w\mathbf{z}}\\
       \Aa_{s_i}^{w\mathbf{z}} & D_i  (s_i w\mathbf{z})
     \end{array}\right)^2 = (v - 1) \left(\begin{array}{cc}
       D_i  (w\mathbf{z}) & \Aa_{s_i}^{s_i w\mathbf{z}}\\
       \Aa_{s_i}^{w\mathbf{z}} & D_i  (s_i w\mathbf{z})
     \end{array}\right) + v \left(\begin{array}{cc}
       1_{\mathcal{M} (w\mathbf{z})} & \\
       & 1_{\mathcal{M} (s_i w\mathbf{z})}
     \end{array}\right) . \]
  This confirms the quadratic relation on the invariant subspace 
  $\mathcal{M}(w\mathbf{z}) \oplus \mathcal{M} (s_i w\mathbf{z})$. Since $\mathfrak{M}
  (\mathbf{z})$ is a direct sum of such spaces, this proves
  (\ref{quadratic}).
  
  We turn to the braid relation (\ref{firstbraid}).  There are a number of cases.  Suppose for example that
  the order of $s_i s_j$ is three. Apply $\mathfrak{T}_i$ to a vector $\phi_{\mathbf{z}} \in \mathcal{M}
  (\mathbf{z})$ which we regard as a direct summand of $\mathfrak{M}
  (\mathbf{z})$. This produces
  \[ \frac{(1 - v) x}{(1 - x)} \phi_{\mathbf{z}} + \Aa_{s_i}^{\mathbf{z}}
     \phi_{\mathbf{z}} \in \mathcal{M} (\mathbf{z}) \oplus \mathcal{M} (s_i \mathbf{z}),
     \qquad x =\mathbf{z}^{\alpha_i^{\vee}} . \]
  Now applying $\mathfrak{T}_j$ gives
  \[ \frac{(1 - v) y}{(1 - y)} \frac{(1 - v) x}{(1 - x)} \phi_{\mathbf{z}} +
     \frac{(1 - v) x y}{(1 - x y)} \Aa_{s_i}^{\mathbf{z}} \phi_{\mathbf{z}}
     + \frac{(1 - v) x}{(1 - x)} \Aa_{s_j}^{\mathbf{z}} \phi_{\mathbf{z}} +
     \Aa_{s_j}^{s_i \mathbf{z}} \Aa_{s_i}^{\mathbf{z}} \phi_{\mathbf{z}}, \]
  where $y =\mathbf{z}^{\alpha_j^\vee}$. Finally one more application of
  $\mathfrak{T}_i$ gives a sum of various terms which we examine individually.
  The term in $\mathcal{M} (\mathbf{z})$ equals $\phi_{\mathbf{z}}$ times the
  following constant
  \[ (1 - v)^3 \left( \frac{1}{x - 1} \right)^2 \frac{1}{y - 1} + \frac{(1 -
     v)}{(x y - 1)} \frac{(1 - v x) (1 - v / x)}{(1 - x) (1 - 1 / x)} . \]
  This is unchanged if we interchange $x$ and $y$, which means that the 
  $\mathcal{M}(\mathbf{z})$ components of $\mathfrak{T}_i \mathfrak{T}_j \mathfrak{T}_i
  \phi_{\mathbf{z}}$ and $\mathfrak{T}_j \mathfrak{T}_i \mathfrak{T}_j
  \phi_{\mathbf{z}}$ are equal. The same is true for the components in $\mathcal{M}
  (s_i \mathbf{z})$, $\mathcal{M} (s_j \mathbf{z})$, $\mathcal{M} (s_i s_j
  \mathbf{z})$, $\mathcal{M}(s_j s_i \mathbf{z})$ and $\mathcal{M} (s_i s_j
  s_i \mathbf{z})$. We leave these for the reader to check, except to note
  that the components in $\mathcal{M} (s_i s_j s_i \mathbf{z})$ are equal by
  (\ref{abraida}). This shows that $\mathfrak{T}_i \mathfrak{T}_j
  \mathfrak{T}_i$ and $\mathfrak{T}_j \mathfrak{T}_i \mathfrak{T}_j$ have the
  same effect on $\phi_{\mathbf{z}} \in\mathcal{M} (\mathbf{z})$. The same argument
  works for $\phi_{w\mathbf{z}} \in\mathcal{M} (w\mathbf{z})$, proving the
  relation~(\ref{firstbraid}) when the order of $s_i s_j$ is three.

  We will leave the case where the order of $s_i s_j$ is two to the
  reader. There remain the cases where the order of $s_i s_j$ is
  four or six. These reduce to the rank two case where the root
  system is of type $C_2$ or $G_2$. It is most convenient to use computer
  algebra for these cases. In Appendix \ref{appendix}, we give a
  \sage\ program to check this.
\end{proof}

As a consequence we have a representation of $\mathcal{H}_v$ on
$\mathfrak{M} (\mathbf{z})$ in which $T_i$ acts through the operator
$\mathfrak{T}_i$. We next show that we may extend this representation
to the affine Hecke algebra.

\begin{theorem}
  \label{hecketheoremaffine}
  The representation of $\mathcal{H}_v$ in Theorem~\ref{hecketheorem}
  extends to a representation of $\widetilde{\mathcal{H}}_v$ in which
  \begin{equation}
    \label{thetaact}
    \theta_{\lambda} (\phi_{w\mathbf{z}}) = (w\mathbf{z})^{\lambda}
    \phi_{w\mathbf{z}} .
    \end{equation}
\end{theorem}

\begin{proof}
  We must check the Bernstein relation. Applying $\theta_{\lambda}
  \mathfrak{T}_i -\mathfrak{T}_i \theta_{s_i \lambda}$ to
  $(\phi_{w\mathbf{z}})$ gives $(\psi_{w\mathbf{z}})$ where
  $\psi_{w\mathbf{z}}$ equals
  \[ (w\mathbf{z})^{\lambda} (\Aa_{s_i}^{s_i w\mathbf{z}} \phi_{s_i
     w\mathbf{z}} + D_i (w\mathbf{z}) \phi_{w\mathbf{z}}) -
     (\Aa_{s_i}^{s_i w\mathbf{z}} (s_i w\mathbf{z})^{s_i \lambda}
     \phi_{s_i w\mathbf{z}} + D_i (w\mathbf{z}) (w\mathbf{z})^{s_i
     \lambda} \phi_{w\mathbf{z}}) . \]
  Note that $(s_i w\mathbf{z})^{s_i \lambda} = (w\mathbf{z})^{ \lambda}$
  so two terms cancel, and
  \begin{equation}
    \label{psifirsteval} \psi_{w\mathbf{z}} = D_i (w\mathbf{z})
    ((w\mathbf{z})^{n \lambda} - (w\mathbf{z})^{n s_i \lambda}) .
  \end{equation}
  This last quantity matches the right-hand side of the Bernstein relation given in (\ref{bernstein}).
\end{proof}

\begin{remark}
As an alternative approach, we could require that for each $\lambda \in
\Lambda$, the action of $\theta_\lambda$ on $\mathfrak{M} (\mathbf{z})$
is defined by
$$ \theta_\lambda \cdot (\phi_{w \mathbf{z}}) = ((w \mathbf{z})^\lambda
\phi_{w \mathbf{z}}), $$ as in Theorem~\ref{hecketheoremaffine}. Then one may
show that the Assumptions~\ref{assofuandme} and subsequent definition of the
action $\mathfrak{T}_i$ in (\ref{tiaction}) are necessary to obtain an
$\widetilde{\mathcal{H}}_v$ action. Indeed, if we assume that the block matrix
$(\Aa_{i,j})$ associated to the action of $T_i$ on the subspace $\mathcal{M}
(\mathbf{z}) \oplus \mathcal{M} (s_i \mathbf{z})$ is composed of linear maps
$\Aa_{i,j}$, then only the diagonal entries of this block matrix affect the
operator $ T_i \theta_\lambda - \theta_{s_i \lambda} T_i$ appearing in the
Bernstein relation. This forces $\Aa_{1,1} = D_{w \mathbf{z}}$ and $\Aa_{2,2} =
D_{s_i w \mathbf{z}}$. Then the quadratic relation necessitates the
identity~(\ref{doublesch}), and the braid relation necessitates the braid
relations on the off-diagonal entries $\Aa_{1,2}$ and $\Aa_{2,1}$ of the block
matrix as given in~(\ref{abraidb}) and (\ref{abraida}).
\end{remark}

\begin{remark} The results in this section would hold if we fixed $\mathbf{z}$ in $\Treg(\C)$
and required only the existence of vector spaces $\mathcal{M}(w\mathbf{z})$ for $w\in W$ satisfying
the conditions given above.  However in the next we will change our viewpoint slightly and vary $\mathbf{z}$.
\end{remark}

\section{Hecke Modules from Whittaker functionals\label{whitheck}}

The Casselman-Shalika formula is a formula for the spherical vector in the
Whittaker model of a principal series representation. However in this section
we will give a different interpretation of it that does not require choosing
any particular vector in the model. We will first use Whittaker functionals to
give an instance of Theorems~\ref{hecketheorem} and~\ref{hecketheoremaffine}.
Then we will extract from this representation the action of the
Hecke algebra on a space of functions on $\Treg(\C)$ by Demazure-Whittaker
operators, which are divided difference operators that are similar to
but slightly different from the well-known Demazure-Lusztig operators. 
We will then project the spherical idempotent in $\widetilde{\mathcal{H}}_v$
into this representation and obtain a version of the Casselman-Shalika
formula. Only after we have done this will we introduce particular
vectors of the representation to reinterpret this formula the usual way.

We preserve the notation presented in the Introduction. Thus $F$ is a
nonarchimedean local field of residue cardinality $q$, $G$ is a split
reductive group defined over $F$, and $\widehat{G}$ is the (connected) Langlands
dual group. We actually require $G$ to be a group scheme over the
ring $\mathfrak{o}$ of integers in $F$ such that $K^{\circ} := G(\mathfrak{o})$ is
a special maximal compact subgroup of $G(F)$, that is, the stabilizer
of a special vertex in the building. This can be arranged using
the integral structure coming from a Chevalley basis. Let $\Phi^{\vee}_+$
denote a set of positive roots of $\Phi^{\vee}$.  Let $B$ denote the
corresponding positive Borel subgroup of $G$ containing the maximal torus $T$
and $U$ be the maximal unipotent subgroup of $B$. Let $\Phi^{\vee}_-$ denote
the complementary set of negative roots, and $B^-$ and $U^-$ be the opposite
Borel subgroup to $B$ and its unipotent radical, respectively. Let $\psi : U^-
(F) \to \mathbb{C}$ be a nondegenerate additive character such that if $x_{-
  \alpha^\vee_i} : F \to U^- (F)$ is the one-parameter subgroup tangent to the
simple root $\alpha^\vee_i$, then $\psi \circ x_{- \alpha^\vee_i}$ is trivial
on the ring $\mathfrak{o}$ of integers in $F$, but no larger fractional ideal.

If $\mathbf{z} \in \Treg(\C) \subseteq \widehat{T} (\mathbb{C})$, then $\mathbf{z}$
indexes an unramified quasicharacter $\chi_{\mathbf{z}}$ of $T (F)$. Indeed,
we may identify $T (F) / T (\mathfrak{o})$ with the weight lattice $\Lambda =
X^{\ast} (\widehat{T})$, so if $t \in T (F)$ let $\lambda$ be the corresponding
weight under this identification, and define $\chi_{\mathbf{z}} (t)
=\mathbf{z}^{\lambda}$.

Let $M (\mathbf{z})$ be the principal series representation induced by
$\chi_{\mathbf{z}}$. This consists of all locally constant functions $f : G
(F) \to \mathbb{C}$ such that
\[ f (bg) = (\delta_B^{1 / 2} \chi_{\mathbf{z}}) (b) f (g), \qquad b \in B (F)\]
where $\delta$ is the modular quasicharacter of $B (F)$.
Let $\pi_{\mathbf{z}} : G (F) \longrightarrow \operatorname{End} (M (\mathbf{z}))$
be the action of $G (F)$ by right translation.
We denote by $\Ai_w:M(\mathbf{z})\to M(\mathbf{wz})$ the
standard intertwining integral
\[(\Ai_w f)(g)=\int_{U(F)\cap wU^-(F)w^{-1}}f(w^{-1}ug)\,d u,\]
modulo convergence issues. (See \cite{CasselmanSpherical}.)
Here we are chosing a representative for the Weyl group element $w$ in
$K^\circ$.

The module $M (\mathbf{z})_{U^-, \psi}$ of coinvariants is one-dimensional
by Rodier~{\cite{RodierWhittaker}}, Theorem~3. It may be identified with $\mathbb{C}$ as
follows. If $f \in M (\mathbf{z})$, define the Whittaker functional
$\omega_{\psi}^{\mathbf{z}} : M (\mathbf{z}) \longrightarrow \mathbb{C}$
to be
\begin{equation}
  \text{} \label{whitint} \int_{U^- (F)} f (u g) \overline{\psi (u)} \, d u.
\end{equation}
The integral is absolutely convergent if $| \mathbf{z}^{\alpha^\vee} | < 1$ for
all positive roots, and may be analytically continued to all $\mathbf{z}$.
See {\cite{CasselmanShalika}}, Section~2. Alternatively, for any
$\mathbf{z}$, the integral (\ref{whitint}) is convergent when $f$ has
compact support in the open Bruhat cell $U^- (F) B (F)$, and can be extended
to all $M (\mathbf{z})$ as in Corollary 1.8 of {\cite{CasselmanShalika}}.

The functional $\omega_{\psi}$ factors through the module of coinvariants, and
hence induces an isomorphism $\omega^{\mathbf{z}}_{\psi} : M
(\mathbf{z})_{U^-, \psi} \longrightarrow \mathbb{C}$ of one-dimensional
vector spaces. Let $v = q^{- 1}$, where $q$ is the cardinality of the residue
field. 

\begin{proposition} \label{magicfactor}
  Let $s_i$ be a simple reflection. We have
  \[ \omega_{\psi}^{s_i \mathbf{z}} \circ \Ai_{s_i}^{\mathbf{z}} =
     \frac{1 - v\mathbf{z}^{- \alpha^\vee}}{1 -\mathbf{z}^{\alpha^\vee}}
     \omega_{\psi}^{\mathbf{z}} . \]
\end{proposition}

\begin{proof}
  See {\cite{CasselmanShalika}}, Proposition 4.3, and
  {\cite{BrubakerBumpLicata}}, Proposition~2.
\end{proof}

Now we are ready to apply Theorems~\ref{hecketheorem} and~
\ref{hecketheoremaffine}. we may take $\mathcal{M}(\mathbf{z})=\mathbb{C}$
which we identify with $M(\mathbf{z})_{U^-,\psi}$ by
means of the Whittaker functionals $\omega_\psi^{\mathbf{z}}$.
Thus $\mathfrak{M}(\mathbf{z})=\bigoplus_{w\in W}\mathbb{C}$ and the
operator $\Aa_{s_i}^\mathbf{z}$ is thus interpreted
as the scalar $\frac{1 - v\mathbf{z}^{- \alpha^\vee}}{1 -\mathbf{z}^{\alpha^\vee}}$.

Let $\mathcal{O}(\widehat{T}(\C))$ denote the space of functions
on $\widehat{T}(\C)$ that are finite linear combinations of
the functions $\mathbf{z}^\lambda$ with $\lambda\in\Lambda$.
Then $\mathcal{O}(\Treg(\C))$ is the localization of this
commutative ring with respect to the multiplicative set
generated by $1-\mathbf{z}^\alpha$ where $\alpha$ runs through
the positive roots. Then $\mathcal{O}(\Treg(\C))$ is a ring of
holomorphic functions on $\Treg(\C)$.

We may think of $\mathfrak{M}$ as a vector bundle over $\Treg(\C)$.
The elements of $\bigoplus_{w\in W}\mathcal{O}(\Treg(\C))$ are
then holomorphic sections. We have Weyl group actions on $\Treg(\C)$ and on
$\mathcal{O}(\Treg(\C))$, so $\mathfrak{M}$ is an equivariant vector bundle.  It
is easy to check that the Weyl group action commutes with the action of
$\widetilde{\mathcal{H}}_v$ defined in Theorems~\ref{hecketheorem}
and~\ref{hecketheoremaffine}.

We will restrict ourselves to the submodule $\mathfrak{M}^W$
of invariant sections.
This vector space is isomorphic to $\mathcal{O}(\Treg(\C))$
as follows. If $(\phi_w)$ is an invariant section,
let $\phi=\phi_1\in\mathcal{O}(\Treg{\C})$. Then $\phi_w=w\phi$
for $w\in W$. If we thus identify $\mathcal{O}(\Treg(\C))$
with $\mathfrak{M}^W$, then it becomes a
$\widetilde{\mathcal{H}}_v$-module.

Let $\rho$ be half the sum of the positive roots. 

\begin{theorem}
  \label{antisphericalthm}
  \begin{enumerate}
  \item[(i)]
    Let $f\in\mathcal{O}(\Treg(\C))$. Identifying $\mathfrak{M}^W$ with $\mathcal{O}(\Treg(\C))$,
the $\widetilde{\mathcal{H}}_v$ action from Section~\ref{schemasect} with 
$\mathcal{A}_{s_i}^{\mathbf{z}} := \displaystyle \frac{1 - v\mathbf{z}^{- \alpha_i^\vee}}{1 -\mathbf{z}^{\alpha_i^\vee}}$
is given by
    \begin{equation}
      \label{demazurewhittaker}
      (T_if)(\mathbf{z})= \frac{(1 - v) \mathbf{z}^{\alpha^\vee_i}}{1 -\mathbf{z}^{\alpha^\vee_i}} f
   (\mathbf{z}) + \frac{1 - v\mathbf{z}^{\alpha^\vee_i}}{1 -\mathbf{z}^{-
          \alpha^\vee_i}} f (s_i \mathbf{z}),
      \qquad \theta_\lambda f(\mathbf{z})=\mathbf{z}^\lambda\,f(\mathbf{z}).
    \end{equation}
  \item[(ii)]
  The module $\mathcal{O}(\Treg(\C))=\mathfrak{M}^W$
  is a free $\theta_{\Lambda}$-module
  generated by the $\mathcal{H}_v$-invariant vector
  $\mathbf{z}^\rho$, which satisfies
  \begin{equation}
    \label{antispherical}
    T_i\mathbf{z}^\rho = -\mathbf{z}^\rho.
  \end{equation}
  \item[(iii)] The representation of $\widetilde{H}_v$ on $\mathcal{O}(\Treg(\C))$
  induces an action on $\mathcal{O}(\widehat{T}(\C))$.
  \end{enumerate}
\end{theorem}

\begin{proof}
  With our identifications, $f$ corresponds to the tuple
  $(\phi_w) \in \mathfrak{M}$ where $\phi_w(\mathbf{z})=f(w^{-1}\mathbf{z})$.
  Applying $T_i$ as in (\ref{tiaction}) with $w=1$ and noting 
  $\mathcal{A}_{s_i}^{s_i \mathbf{z}} = \frac{1 - v\mathbf{z}^{\alpha^\vee_i}}{1 -\mathbf{z}^{-
          \alpha^\vee_i}}$,
  we obtain (\ref{demazurewhittaker}). In particular if
  $f(\mathbf{z})=\mathbf{z}^\rho$,
  this equals
  \[ \frac{(1 - v) \mathbf{z}^{\alpha^\vee_i}}{1 -\mathbf{z}^{\alpha^\vee_i}}
  \mathbf{z}^{\rho} +\Aa_{s_i}^{s_i \mathbf{z}}
     \mathbf{z}^{s_i \rho} = \frac{(1 - v) \mathbf{z}^{\alpha^\vee_i}}{1
     -\mathbf{z}^{\alpha^\vee_i}} \mathbf{z}^{\rho} + \frac{1 -
     v\mathbf{z}^{\alpha^\vee_i}}{1 -\mathbf{z}^{- \alpha^\vee_i}} \mathbf{z}^{\rho -
     \alpha^\vee_i} . \]
  After some algebra, this equals $-\mathbf{z}^{\rho}$, proving
  (\ref{antispherical}).
  
  It is clear from (\ref{thetaact}) that $\theta_{\lambda}
  \mathbf{z}^{\rho}$ where $\lambda$ runs through $\Lambda$ are a
  basis of $\mathfrak{M}^W$, and (ii) is now clear.

  As for (iii), $\mathcal{O}(\widehat{T}(\C))$ is contained in $\mathcal{O}(\Treg(\C))$,
  and we must argue that it is invariant under the $T_i$, which \textit{a priori}
  could introduce denominators. To check that the numerator in $T_if$
  is always divisible by $1-\mathbf{z}^{\alpha_i^\vee}$ it is sufficient to consider the case
  $f(\mathbf{z})=\mathbf{z}^\lambda$, since these span $\mathcal{O}(\widehat{T}(\C))$.
  Then $f(s_i\mathbf{z})=f(\mathbf{z})\,\mathbf{z}^{-\langle\alpha_i,\lambda\rangle\alpha_i^\vee}$
  and $\langle\alpha_i,\lambda\rangle\in\Z$, from which the required
  divisibility is easily checked.
\end{proof}

The {\textit{antispherical}} module of $\tilde{\mathcal{H}}_v$ is the module
induced from the character $T_i \mapsto - 1$ of $\mathcal{H}_v$.

\begin{corollary}
  \label{anticonv}The representation of Theorem~\ref{antisphericalthm} 
  on $\mathcal{O}(\widehat{T}(\C))$ is the antispherical representation.
\end{corollary}

\begin{proof}
  This follows since
  $\widetilde{\mathcal{H}}_v=\theta_{\Lambda}\otimes\mathcal{H}_v$,
  and
  $\mathcal{O}(\widehat{T}(\C))=\bigoplus_{\lambda\in\Lambda}\theta_\lambda(\mathbf{z}^\rho)$.
\end{proof}

The antispherical representation appears in another way in this subject:
Arkhipov and Bezrukavnikov~{\cite{ArkhipovBezrukavnikov}} and Chan and
Savin~{\cite{ChanSavin}} showed that it is isomorphic to the convolution
module of Iwahori fixed vectors in the Gelfand-Graev representation
$\operatorname{ind}^G_{U^-} (\psi)$. Although that result and
Corollary~\ref{anticonv} both relate Whittaker models to the antispherical
representation, the two representations arise in different ways, and the two
statements are not obviously equivalent. On the other hand, the antispherical
representation appears in this connection
in~\cite{BrubakerBumpFriedbergMatrix, BrubakerBumpLicata, BrubakerBumpFriedbergPacificJ}, and the results there are closely
related to Theorem~\ref{antisphericalthm} and its corollary.

\begin{remark}
\label{antisphericalremark}
For reasons that we will explain shortly we wish to consider
the action of Theorem~\ref{antisphericalthm} conjugated by the endomorphism $f
\mapsto f'$ of $\mathcal{O} (\widehat{T}(\C))$, where
$f' (\mathbf{z}) = f (\mathbf{z}^{- 1})$. In this action $T_i$ acts by the operator
\begin{equation}
  \label{realdemazurewhittaker} \mathcal{T}_i\,f (\mathbf{z}) = \frac{1 -
  v}{\mathbf{z}^{\alpha^\vee_i} - 1} f (\mathbf{z}) + \frac{v\mathbf{z}^{-
      \alpha^\vee_i} - 1}{\mathbf{z}^{\alpha^\vee_i} - 1} f (s_i \mathbf{z})=
  \frac{f (\mathbf{z}) - f (s_i
   \mathbf{z})}{\mathbf{z}^{\alpha_i^{\vee}} - 1} - v \frac{f
   (\mathbf{z}) -\mathbf{z}^{- \alpha_i^{\vee}} f (s_i
   \mathbf{z})}{\mathbf{z}^{\alpha_i^\vee} - 1},
\end{equation}
which agrees with the operator $\mathfrak{T}_i$ defined in
{\cite{BrubakerBumpLicata}}, equations (3) and (14). We refer to this as the
{\textit{Demazure-Whittaker operator}}. Also in the modified action
\[\theta_\lambda f(\mathbf{z})=\mathbf{z}^{-\lambda}f(\mathbf{z}).\]
\end{remark}

The Casselman-Shalika formula is a statement about the spherical vector in the
Whittaker model of $M (\mathbf{z})$, and is usually proved by considerations
involving the Iwahori-fixed vectors in the model. However, as we will now
explain, we may use our present setup to formulate a version of the
Casselman-Shalika formula that does not involve either the spherical vector or
the Iwahori fixed vectors.

The finite Hecke algebra $\mathcal{H}_v$ generated by the $T_i$ has 
a basis $T_w = T_{i_1} \cdots T_{i_k}$ ($w\in W$), where $w = s_{i_1} \cdots
s_{i_k}$ is a reduced decomposition into simple reflections. Then $I^{\circ} =
\sum_w T_w$ is an idempotent that generates a 2-sided ideal in
$\widetilde{\mathcal{H}}_v$. This ideal is the commutative spherical Hecke
algebra $\mathcal{H}^{\circ}_v$, and $I^\circ$ is its unit
element. Applied to $M (\mathbf{z})$ in the convolution action, the operator
$I^{\circ}$ is the projection onto the spherical vector. This point of view
leads to the classical Casselman-Shalika formula for the spherical
vector in the Whittaker model.

However we wish to do something different and apply $I^\circ$ in the action of
Remark~\ref{antisphericalremark}. Hence we get a version of the Casselman-Shalika
formula that makes no reference to spherical vectors. This fact is
equivalent to equation (5.5.14) in Macdonald~\cite{MacdonaldOrthogonal}.
Here we will prove it again using our methods.

\begin{theorem}
  \label{abstractcs}
  Let $\lambda \in \Lambda$ be a dominant weight, and let
  $\chi_{\lambda}$ be the irreducible character of $\widehat{G} (\mathbb{C})$
  with highest weight $\lambda$. Let $\mathcal{I}_{\circ}$ denote the action
  of $I^{\circ} \in \tilde{\mathcal{H}}_v$ on $\mathcal{O} (\widehat{T}(\C))$ in the
  modified action of Remark~\ref{antisphericalremark}. Then
  \begin{equation}
    \label{abstractcsfo}
    \mathcal{I}_{\circ} (\mathbf{z}^{\lambda}) = \prod_{\alpha^\vee \in \Phi^{\vee}_+} (1
    - v\mathbf{z}^{- \alpha^\vee}) \chi_{\lambda} (\mathbf{z}) .
  \end{equation}
\end{theorem}

\begin{proof}
  Let $R$ be the localization of the ring $\mathcal{O} (\widehat{T}(\mathbb{C}))$
  obtained by adjoining the
  inverses of $1 -\mathbf{z}^{\alpha^{\vee}}$ and also $1 -
  v\mathbf{z}^{\alpha^{\vee}}$ as $\alpha^{\vee}$ runs through the roots.
  We will work in the twisted group ring $\mathcal{E}$ of $W$ over $R$.
  Thus $\mathcal{E}$ is the free $R$-module spanned by $W$, and if $f \in R$,
  $w \in W$ then $w f w^{- 1} = {^w f}$. Elements of $\mathcal{E}$ may be
  interpreted as operators on meromorphic functions on $\widehat{T}(\mathbb{C})$
  in the obvious way.
  
  If $w \in W$ let $\mathcal{T}_w =\mathcal{T}_{i_1} \cdots \mathcal{T}_{i_k}$
  where $w = s_{i_1} \cdots s_{i_k}$ is a reduced expression, and
  $\mathcal{T}_i$ is the operator defined by (\ref{realdemazurewhittaker}).
  Thus we may regard $\mathcal{T}_i$ as the element
  \[ \mathcal{T}_i = \frac{1 - v}{\mathbf{z}^{\alpha^{\vee}_i} - 1} \cdot 1_W
     + \frac{v \mathbf{z}^{- \alpha^{\vee}_i} -
     1}{\mathbf{z}^{\alpha^{\vee}_i} - 1} \cdot s_i \]
  of $\mathcal{E}$. Let
  \[ \mathcal{I}^\circ = \sum_{w \in W} \mathcal{T}_w, \qquad \widetilde{\mathcal{I}}^\circ = \prod_{\alpha \in
     \Phi^+} (1 - v\mathbf{z}^{- \alpha})^{- 1} \mathcal{I}^\circ . \]

  Let us show that
  \begin{equation}
    \label{gamprewinv} w \widetilde{\mathcal{I}}^\circ = \widetilde{\mathcal{I}}^\circ
  \end{equation}
  for $w \in W$. It is sufficient to show this when $w = s_i$ is a simple
  reflection. Remembering that $s_i$ sends $\alpha_i^{\vee}$ to $-
  \alpha_i^{\vee}$ and permutes the remaining positive roots, this is
  equivalent to
  \begin{equation}
    \label{gamtranspro} s_i \mathcal{I}^\circ = \left( \frac{1 -
    v\mathbf{z}^{\alpha_i^{\vee}}}{1 - v\mathbf{z}^{- \alpha_i^{\vee}}}
    \right) \mathcal{I}^\circ .
  \end{equation}
  To prove this, we write
  \[ \mathcal{I}^\circ = \sum_{\scriptsize{\begin{array}{c}
       w \in W\\
       s_i w > w
     \end{array}}} (1 +\mathcal{T}_i) \mathcal{T}_w . \]
  From this, it is enough to prove that
  \begin{equation}
    \label{dtranspro} s_i (1 +\mathcal{T}_i) = \left( \frac{1 -
    v\mathbf{z}^{\alpha_i^{\vee}}}{1 - v\mathbf{z}^{- \alpha_i^{\vee}}}
    \right) (1 +\mathcal{T}_i) .
  \end{equation}
  Now
  \[ (1 +\mathcal{T}_i) = (1 - v\mathbf{z}^{- \alpha_i^{\vee}}) \partial_i,
     \qquad \partial_i = (1 -\mathbf{z}^{- \alpha_i^{\vee}})^{- 1} (1
     -\mathbf{z}^{- \alpha_i} s_i) \]
  and $s_i \partial_i = \partial_i$, from which (\ref{dtranspro}) and hence
  (\ref{gamtranspro}) may be checked. This proves (\ref{gamprewinv}).
  
  Let us write $\widetilde{\mathcal{I}}^\circ = \sum_{w \in W} \phi_w \cdot w$ with $\phi_w \in R$.
  By (\ref{gamprewinv}), we have $\phi_w = {^w \phi}$ where $\phi = 1_W$. 
  Let $w_0$ be the long Weyl group element. We will show that
  \begin{equation}
    \label{phiwopro} \phi_{w_0} = \prod_{\alpha \in \Phi^+} (1
    -\mathbf{z}^{\alpha_i^{\vee}})^{- 1} .
  \end{equation}
  To this end let us therefore compute the coefficient of $w_0$ in $\mathcal{I}^\circ$.
  Each $\mathcal{T}_w =\mathcal{T}_{i_1} \cdots \mathcal{T}_{i_k}$ can be
  multiplied out, selecting from each factor $\mathcal{T}_i$ 
  either $\frac{1 - v}{\mathbf{z}^{\alpha^{\vee}_i}
  - 1} \cdot 1_W$ or $\frac{v \mathbf{z}^{- \alpha^{\vee}_i} -
  1}{\mathbf{z}^{\alpha^{\vee}_i} - 1} \cdot s_i$. The only way we may get a
  contribution of $w_0$ is to take $w = w_0$ and select the second
  contribution each time. Denoting $H (\alpha^{\vee}) = \frac{v\mathbf{z}^{-
  \alpha^{\vee}} - 1}{\mathbf{z}^{\alpha^{\vee}} - 1}$ the contribution is
  \[ H (\alpha_{i_1}^{\vee}) H (s_{i_1} \alpha^{\vee}_{i_2}) H (s_{i_1}
     s_{i_2} \alpha^{\vee}_{i_3}) \cdots \; . \]
  But $\alpha^{\vee}_{i_1}, s_{i_1} (\alpha^{\vee}_{i_2}), \cdots$ is an
  enumeration of the positive roots. So this product is
  \[ \prod_{\alpha \in \Phi^+} \frac{1 - v\mathbf{z}^{- \alpha_i^{\vee}}}{1
     -\mathbf{z}^{\alpha_i^\vee}} . \]
  This is the coefficient in $\mathcal{I}^\circ$, and for $\widetilde{\mathcal{I}}^\circ$ we discard the
  numerator to obtain (\ref{phiwopro}). Therefore
  \[ \phi = \prod_{\alpha \in \Phi^+} (1 -\mathbf{z}^{- \alpha_i^{\vee}})^{-
     1}, \]
  \[ \mathcal{I}^\circ = \prod_{\alpha \in \Phi^+} (1 - v\mathbf{z}^{-
     \alpha_i^{\vee}}) \widetilde{\mathcal{I}}^\circ = \prod_{\alpha \in \Phi^+} (1 -
     v\mathbf{z}^{- \alpha_i^{\vee}}) \sum_{w \in W} { \left(
     \prod_{\alpha \in \Phi^+} (1 -\mathbf{z}^{- w\alpha_i^{\vee}})^{- 1}
     \right) w.} \]
  Thus if $\lambda$ is a dominant weight,
  \[ \mathcal{I}^\circ \mathbf{z}^{\lambda} = \prod_{\alpha \in \Phi^+} (1 -
     v\mathbf{z}^{- \alpha_i^{\vee}}) \left[ \sum_{w \in W} { \left(
     \prod_{\alpha \in \Phi^+} (1 -\mathbf{z}^{- w\alpha_i^{\vee}})^{- 1}
     \right) \mathbf{z}^{w \lambda}} \right], \]
  and using the Weyl character formula, the expression in brackets is
  $\chi_{\lambda} (\mathbf{z})$.
\end{proof}

We next explain how Theorem~\ref{abstractcs} is related to the
Casselman-Shalika formula. The right hand side of (\ref{abstractcsfo})
agrees with the Casselman-Shalika formula for the spherical vector in
the Whittaker model, but we have refrained from introducing particular vectors 
in the model.

To make the connection, we must work with the spherical and other 
Iwahori fixed vectors, and we must discuss our conventions. The conventions
chosen in Brubaker, Bump and Licata~{\cite{BrubakerBumpLicata}} were different from those in
Casselman and Shalika~{\cite{CasselmanShalika}}.\footnote{We are referring to the version
of~\cite{BrubakerBumpLicata} that was published in the Journal of Number
Theory -- the version in the arxiv uses conventions closer to
{\cite{CasselmanShalika}}.}  The conventions in {\cite{BrubakerBumpLicata}}
are chosen to keep the long Weyl group element $w_0$ out of the
formulas. This seems advantageous for reasons of elegance, and also for the
purposes of generalizing to the Kac-Moody case, where the Weyl group may be
infinite and have no longest element. To follow the conventions of
{\cite{BrubakerBumpLicata}}:
\begin{itemize}
  \item Principal series representations are induced, as usual, from the
  positive Borel subgroup $B$. However instead of the representation $M
  (\mathbf{z})$ we compute the Whittaker model of its contagradient $M
  (\mathbf{z}^{- 1})$.
  
  \item We take Whittaker models with respect to the negative unipotent
  subgroup $U^-$.
  
  \item We consider vectors fixed with respect to the negative Iwahori
  subgroup $J^-$. This is the preimage in the special maximal compact subgroup
  $G (\mathfrak{o})$ of the negative Borel subgroup $B^- (\mathbb{F}_q)$
  under the reduction mod $\mathfrak{p}$ map $G (\mathfrak{o}) \longrightarrow
  G (\mathbb{F}_q)$.
\end{itemize}

Let us adopt these same choices. In order to pass to the
contragredient, we wish to work with $\mathfrak{M} (\mathbf{z}^{- 1})$
instead of $\mathfrak{M} (\mathbf{z})$. This explains why we
conjugated by the map $f\mapsto f'$ in Remark~\ref{antisphericalremark}.
Following these conventions, the right-hand side of (\ref{abstractcsfo})
agrees with the Casselman-Shalika formula, even though we have not
yet made any reference to the spherical vector or indeed any Iwahori
fixed-vector in the Whittaker model.

Let us define particular $J^-$ fixed vectors in $M
(\mathbf{z}^{- 1})$. These vectors are parametrized by Weyl group elements and
are given by the formula
\[ \Phi_w^{\mathbf{z}^{- 1}} (b w' k) = \left\{ \begin{array}{ll}
     (\delta^{1 / 2} \chi_{\mathbf{z}^{- 1}}) (b) & \text{if $w = w'$,}\\
     0 & \text{otherwise,}
   \end{array} \right. \]
for $b \in B (F)$, $w$ and $w'$ Weyl group elements, and $k \in J^-$. Here, by
abuse of notation, we use the same notation $w'$ for a representative of the
Weyl group element $w'$, chosen in $G (\mathfrak{o})$. The definition does not
depend on this choice because $\chi_{\mathbf{z}^{- 1}}$ is unramified.

We have noted that $T (F) / T (\mathfrak{o}) \cong \Lambda$. If
$\lambda \in \Lambda$, let $t_{\lambda}$ denote a representative in $T
(F)$ of the coset mod $T (\mathfrak{o})$ corresponding to $\lambda$ under this
isomorphism.

\begin{proposition}
  \label{phioneiwahori}
  We have
  \[ \omega_{\psi}^{\mathbf{z}^{- 1}} (\pi_{\mathbf{z}^{- 1}} (t_{-
     \lambda}) \Phi_1^{\mathbf{z}^{- 1}}) = \left\{ \begin{array}{ll}
       \delta^{1 / 2} (t_{\lambda})\,\mathbf{z}^{\lambda} & \text{if
       $\lambda$ is dominant,}\\
       0 & \text{otherwise.}
     \end{array} \right. \]
\end{proposition}

\begin{proof}
  This is Proposition~1 in {\cite{BrubakerBumpLicata}}.
\end{proof}

Our point of view in Theorem~\ref{antisphericalthm} was that
$\mathbf{z}^\lambda$ is a $W$-invariant section of a vector bundle
$\mathfrak{M}(\mathbf{z})$. So in Remark~\ref{antisphericalremark},
it is a $W$-invariant section of the vector bundle
$\mathfrak{M}(\mathbf{z}^{-1})$. It is to be remembered
that $\mathcal{M}(\mathbf{z}^{-1})$ is the one-dimensional space of
coinvariants for the Whittaker model; it is a quotient of
the principal series representation $M(\mathbf{z}^{-1})$.
Project the section $\mathbf{z}^\lambda$ onto the summand
$\mathcal{M}(\mathbf{z}^{-1})$. If $\lambda$ is dominant,
Proposition~\ref{phioneiwahori} gives a way of lifting this projection
to $M(\mathbf{z}^{-1})$, namely we take the representative
$\delta^{-1/2}(t_\lambda)\pi_{\mathbf{z}^{-1}}(t_{-\lambda}\Phi_1^{\mathbf{z}^{-1}})$.

Now let $\Phi_\circ^{\mathbf{z}^{-1}}$ be the spherical vector
\[\Phi_\circ^{\mathbf{z}}(b k)=
(\delta^{1 / 2} \chi_{\mathbf{z}^{- 1}}) (b)\]
for $b\in B(F)$ and $k\in G(\mathfrak{o})$. We may now
deduce the usual form of the Casselman-Shalika formula
from Theorem~\ref{abstractcs}. The proof depends on a
calculation from~\cite{BrubakerBumpLicata} that in
turn depends on results of \cite{CasselmanSpherical}.

\begin{theorem}[\cite{CasselmanShalika}]
  Let $\lambda$ be a dominant weight. Then
  \[\omega_{\psi}^{\mathbf{z}^{-1}}(\pi_{\mathbf{z}^{-1}}(t_{-\lambda})\Phi_\circ^{\mathbf{z}^{-1}})=
  \delta^{1/2}(t_\lambda)\,\prod_{\alpha^\vee \in \Phi^{\vee}_+}\left(1
    - v\mathbf{z}^{- \alpha^\vee}\right)
\chi_{\lambda}(\mathbf{z}).\]
\end{theorem}

\begin{proof}
  If $w\in W$ let $\mathcal{T}_w$ denote the image of $T_w$ in the
  representation of Remark~\ref{antisphericalremark}, so that
  $\mathcal{T}_{s_i}=\mathcal{T}_i$. It follows from 
  Theorem~1 of \cite{BrubakerBumpLicata} that if $s_iw>w$
  then
  \[
  \omega_{\psi}^{\mathbf{z}^{-1}}(\pi_{\mathbf{z}^{-1}}(t_{-\lambda})\Phi_{s_iw}^{\mathbf{z}^{-1}})=
  \mathcal{T}_i\,\omega_{\psi}^{\mathbf{z}^{-1}}(\pi_{\mathbf{z}^{-1}}(t_{-\lambda})\Phi_{w}^{\mathbf{z}^{-1}}).
  \]
  Therefore
  \[
  \omega_{\psi}^{\mathbf{z}^{-1}}(\pi_{\mathbf{z}^{-1}}(t_{-\lambda})\Phi_{w}^{\mathbf{z}^{-1}}) =
  \mathcal{T}_w\,\omega_{\psi}^{\mathbf{z}^{-1}}(\pi_{\mathbf{z}^{-1}}(t_{-\lambda})\Phi_{1}^{\mathbf{z}^{-1}})
  \]
  and since $\mathcal{I}^\circ=\sum_w\mathcal{T}_w$ and
  $\Phi_\circ^{\mathbf{z}^{-1}}=\sum_w\Phi_w^{\mathbf{z}^{-1}}$ it follows
  that 
  \[\mathcal{I}^\circ\omega_{\psi}^{\mathbf{z}^{-1}}(\pi_{\mathbf{z}^{-1}}(t_{-\lambda})\Phi_{1}^{\mathbf{z}^{-1}})
  =\omega_{\psi}^{\mathbf{z}^{-1}}(\pi_{\mathbf{z}^{-1}}(t_{-\lambda})\Phi_{\circ}^{\mathbf{z}^{-1}}).\]
  Applying Theorem~\ref{abstractcs} and Proposition~\ref{phioneiwahori}, the result follows.
\end{proof}

We emphasize that the principles of this section are applicable to other functionals. Let us
consider the {\textit{spherical functional}} $\sigma_{\mathbf{z}} : M
(\mathbf{z}) \longrightarrow \mathbb{C}$ defined by
\[ \sigma_{\mathbf{z}} (\phi) = \int_{K^{\circ}} \phi (k) \, d k. \]
Analogous to Proposition~\ref{magicfactor}, we have
\[ \sigma_{s_i \mathbf{z}} \circ \Ai_{s_i}^{\mathbf{z}} = \frac{1
   - v\mathbf{z}^{\alpha_i^{\vee}}}{1 -\mathbf{z}^{\alpha_i^{\vee}}}
   \sigma_{\mathbf{z}} . \]
This may be deduced from Theorem 3.1 of~{\cite{CasselmanSpherical}}.

Now we make statements analogous to those for the Whittaker functional in Theorem~\ref{antisphericalthm}. Thus
let $\mathcal{M} (\mathbf{z})$ be the module of coinvariants for the
spherical functional and construct the module $\mathfrak{M} (\mathbf{z})$,
which we regard as the module of sections of a vector bundle over
$\Treg(\mathbb{C})$. As before, this bundle is equivariant
for the action of the Weyl group, and if we take the module of invariant
sections, these extend to all of $\hat{T} (\mathbb{C})$. We thus obtain the
representation of $\tilde{\mathcal{H}}_v$ on $\mathcal{O} (\hat{T}
(\mathbb{C}))$ in which
\[ T_i f (\mathbf{z}) = \frac{(1 - v) \mathbf{z}^{\alpha^{\vee}_i}}{1
   -\mathbf{z}^{\alpha_i^{\vee}}} f (\mathbf{z}) + \frac{1 -
   v\mathbf{z}^{- \alpha_i^{\vee}}}{1 -\mathbf{z}^{-\alpha_i^{\vee}}} f(s_i \mathbf{z}),
   \qquad \theta_{\lambda} f (\mathbf{z})=\mathbf{z}^{\lambda} f (\mathbf{z}) . \]
The constant function $f (\mathbf{z}) = 1$ is now invariant under the $T_i$,
and $T_i \cdot 1 = v$. So this representation is induced from the
one-dimensional trivial module corresponding to the character $T_w \mapsto
v^{\ell (w)}$ of $\mathcal{H}_v$, where $\ell$ is the length function on the
Weyl group.

Conjugating this representation by the map $\mathbf{z} \mapsto
-\mathbf{z}$ as in the previous section gives the action
\begin{equation}
  \label{demazurelusztig} T_i f (\mathbf{z}) = \frac{f (\mathbf{z}) - f
  (s_i z)}{\mathbf{z}^{\alpha_i^{\vee}} - 1} - v \frac{f (\mathbf{z})
  -\mathbf{z}^{\alpha_i^{\vee}} f (s_i
  \mathbf{z})}{\mathbf{z}^{\alpha_i^{\vee}} - 1}, \qquad \theta_{\lambda}
  f (\mathbf{z}) =\mathbf{z}^{- \lambda} f (\mathbf{z}) .
\end{equation}
This is the representation defined by Lusztig~{\cite{LusztigEquivariant}}, and the
operators (\ref{demazurelusztig}) are called {\textit{Demazure-Lusztig
operators}}. See Ion~\cite{IonMatrix} and Cherednik and Ma~\cite{CheredikMaI, CherednikMaII}
for connections with the double affine Hecke algebra and nonsymmetric Macdonald polynomials.

\section{Hecke Modules from $R$-matrices}\label{HeckeRM}

{\textit{Schur-Weyl duality}} refers to the commuting representations of the
symmetric group $S_r$, or its group algebra,
and thes general linear group $\GL(n,\C)$, or its enveloping algebra
$U (\mathfrak{g}\mathfrak{l} (n))$, on $\otimes^r V$ where $V$ is a $\GL(n,\C)$-module. 
In this section $V$ is the standard module $\mathbb{C}^n$. This tensor product representation is multiplicity-free in a strong sense and
\[\otimes^r V\cong \oplus_\lambda \,\,\pi_\lambda^{S_r}\otimes\pi_\lambda^{\GL(n)}\]
as $S_r\otimes\GL(n,\C)$-modules, where $\lambda$ runs through
the partitions of $r$ of length $\leqslant n$. Here $\pi_\lambda^{S_r}$ is
the irreducible Specht module of $S_r$ for the partition $\lambda$,
and $\pi_\lambda^{\GL(n)}$ is the irreducible $\GL(n,\C)$-module
with highest weight~$\lambda$.

The Hecke algebra (of type A) $\mathcal{H}_v$ can be thought of as a deformation of the group algebra of the symmetric group $S_r$. One can also deform $U (\mathfrak{g}\mathfrak{l} (n))$ to obtain 
$U_q (\mathfrak{g}\mathfrak{l} (n))$, a quasitriangular Hopf algebra with the
same Chevalley generators as $U (\mathfrak{g}\mathfrak{l} (n))$ subject to the
$quantum$ Serre relations. The definition of $U_{\sqrt{v}} (\mathfrak{g}\mathfrak{l} (n))$ is standard and can be found in chapter 9 in Chari and Pressley~\cite{ChariPressleyBook}, so we will not mention it. We do point out that when $q$ is a root of unity, we use Lusztig's quantum group definition. We will also consider the affine quantum group $U_{\sqrt{v}}(\widehat{\mathfrak{gl}}(n))$ in this paper, which is a quantization of a central extension of the loop algebra of $\mathfrak{gl}(n)$; for its formal definition see Section 12.2 in \cite{ChariPressleyBook}.  

Jimbo {\cite{JimboHecke}} generalized Schur-Weyl duality by showing
that $\mathcal{H}_v$ and $U_{\sqrt{v}} (\mathfrak{g}\mathfrak{l}
(n))$ also have commuting representations on $\otimes^r V$. Here we
have taken the deformation parameter called $q$ in \cite{JimboHecke} to be a
square root of a parameter $v$, for later comparison with
the results of Theorems~\ref{hecketheorem} and~\ref{hecketheoremaffine}.

To explain Jimbo's generalization of Schur-Weyl duality,
 we recall some solutions of the Yang-Baxter equation,
related to the $R$-matrices for the quantum groups
$U_{\sqrt{v}}(\mathfrak{g}\mathfrak{l} (n))$ and
$U_{\sqrt{v}}(\widehat{\mathfrak{g}\mathfrak{l}} (n))$.
If $n = 2$, these are due to Baxter; for general $n$, they are due to
Jimbo. Let $M = \Mat_n (\mathbb{C})$. Let $e_{i j}$ be the elementary $n \times n$
matrix with 1 in the $(i, j)$-th position, and $0$'s elsewhere. Let
\begin{equation} \label{glnRmatrix} R = \sum_i {\sqrt{v}} \, e_{i i} \otimes e_{i i} + \sum_{i \neq j} e_{i i} \otimes
   e_{j j} + \bigl({\sqrt{v}} -\frac1{\sqrt{v}}\bigr) \sum_{i > j} e_{i j} \otimes e_{j i} 
\end{equation}
in $M \otimes M$. This is a solution to the Yang-Baxter equation
\begin{equation}
  \label{glnybe} R_{12} R_{13} R_{23} = R_{23} R_{13} R_{12},
\end{equation}
where the meaning of the subscripts is as follows. Regard $R$ as an
endomorphism of $V \otimes V$. 
Then $R_{i j}$ is the
endomorphism of $V \otimes V \otimes V$ in which the first tensor factor of
$R_{i j}$ acts on the $i$-th tensor factor of $V \otimes V \otimes V$, and the
second tensor factor acts on the $j$-th tensor factor, with the identity
matrix acting on the remaining copy of $V$.

The vector space $V$ may be regarded as the standard $n$-dimensional
module in the tensor category of $U_{\sqrt{v}} (\mathfrak{g}\mathfrak{l} (n))$-modules.
Then the map $\tau R :V \otimes V \longrightarrow V \otimes V$ is a morphism in
this category, where $\tau:V\otimes V\to V\otimes V$
is the map $\tau (x \otimes y) = y \otimes x$. The Yang-Baxter equation (\ref{glnybe}) is a reflection of
the braidedness of the category. See {\cite{Drinfeld}}. Let $1\leqslant
i\leqslant r-1$, and let $(\tau R)_{i,i+1}$ denote the endomorphism
$I_{\otimes^{i-1}V}\otimes (\tau R)\otimes I_{\otimes^{r-1-i}V}$ of $\otimes^r V$.

\begin{proposition}[\cite{JimboHecke}]
  \label{JimboHeckeTheorem}
  We have commuting actions of
  $U_{\sqrt{v}}
  (\mathfrak{g}\mathfrak{l} (n))$ and $\mathcal{H}_v$ on $\otimes^r V$
  in which $T_i$ acts by $\sqrt{v}\,(\tau R)_{i, i + 1}$.
\end{proposition}

\begin{proof}
  In keeping with the subscript convention explained above, $(\tau R)_{i, i +
  1}$ is $\tau R$ applied to the $i$ and $i + 1$ copies of $V$, with the
  identity transformation on the other copies. Since this is an intertwining
  operator of $U_{\sqrt{v}} (\mathfrak{g}\mathfrak{l} (n))$-modules, it commutes with
  the action of $U_{\sqrt{v}} (\mathfrak{g}\mathfrak{l} (n))$, but we need to check
  that the $T_i$ satisfy the defining relations for $\mathcal{H}_v$.
  The quadratic relation
  \[T^2 = (v-1)T + v\]
  for
  \[ T = \sqrt{v}\,\tau R = \sum_i v \, e_{i i} \otimes e_{i i} + \sqrt{v} \sum_{i \neq j} e_{j i}
     \otimes e_{i j} + (v-1) \sum_{i > j} e_{j j} \otimes e_{i i} \]
  is easily checked, and the braid relation when $n(i,j)=3$ is the Yang-Baxter equation.
\end{proof}

The affine quantum group $U_{\sqrt{v}}(\widehat{\mathfrak{g}\mathfrak{l}} (n))$ 
has an evaluation standard module $V_x$ for every $x \in \mathbb{C}^{\times}$. As a vector space, $V_x$ is the same as $V$ (in particular we choose the same basis for $V_x$ as we do for $V$). If $x \in \mathbb{C}^{\times}$, let 
\begin{equation}\label{glnhatRmatrix}
R (x) = R - x R_{21}^{- 1},
\end{equation}
where $R_{21}:V\otimes V\to V\otimes V$ acts by $\tau R\tau$ on the indicated copies of $V$.
Then Jimbo
{\cite{JimboToda}} proved that for each $x \in \mathbb{C}^{\times}$, $\tau R(xy^{-1}) : V_{x} \otimes V_y \to V_{y} \otimes V_{x}$ is a $U_{\sqrt{v}}(\widehat{\mathfrak{g}\mathfrak{l}} (n))$-module homomorphism and that $R(x)$ is a solution to the {\it parametrized} Yang-Baxter equation:
\begin{equation}
  \label{glnhatybe} R_{12} (x) R_{13} (x y) R_{23} (y) = R_{23} (y) R_{13} (x
  y) R_{12} (x).
\end{equation}

\begin{remark}
  The linear expression (\ref{glnhatRmatrix}) for the affine $R$-matrix $R(x)$ in terms of $R$
  is only valid for Cartan type~$A$. See~\cite{JimboToda} for expressions for
  the other affine types that are quadratic or quartic in~$x$.
\end{remark}

The endomorphism $\tau R(xy^{-1})\,\tau R((xy^{-1})^{-1})$ of $V_y \otimes V_x$ is a scalar. More
precisely,

\begin{equation}
  \label{doubler}
  \tau R((xy^{-1})\circ \tau R(((xy^{-1})^{-1})=\bigl(\sqrt{v}-\frac{xy^{-1}}{\sqrt{v}}\bigr)\bigl(\sqrt{v}-\frac{1}{xy^{-1}\sqrt{v}}\bigr)\,.
\end{equation}

We may now give an example of the situation in
Theorems~\ref{hecketheorem} and~\ref{hecketheoremaffine}.
We take $\mathcal{M}(w\mathbf{z})$ to be $V_{w(z_1)} \otimes \cdots \otimes V_{w(z_r)}$ and the map $\Aa_{s_i}^{w\mathbf{z}}: \mathcal{M}(w\mathbf{z})\to \mathcal{M}(s_iw\mathbf{z})$
to be
\begin{equation}
  \label{baxexamplea}
  \Aa_{s_i}^{w\mathbf{z}}=\frac{\sqrt{v}}{1-(w\mathbf{z})^{\alpha_i^\vee}}(\tau
  R(w\mathbf{z}^{\alpha_i^\vee}))_{i,i+1}.
\end{equation}

\begin{theorem}
  \label{heckeybeinstance}
  With these data, the assumptions of
  Theorems~\ref{hecketheorem} and~\ref{hecketheoremaffine}
  are satisfied. The resulting $\mathfrak{M}(\mathbf{z})$
  has commuting actions of $\widetilde{H}_v$ and
  $U_{\sqrt{v}}(\widehat{\mathfrak{gl}}(n))$.
\end{theorem}

\begin{proof}
To check the hypotheses of
Theorems~\ref{hecketheorem} and~\ref{hecketheoremaffine},
we note that the braid relation (\ref{abraida}) follows from the Yang-Baxter equation
(\ref{glnhatybe}).
The constant in front of (\ref{baxexamplea}) is chosen to
make $\Aa^{s_i \mathbf{z}}_{s_i} \circ \Aa_{s_i}^{\mathbf{z}}$ equal the
constant (\ref{doublesch}). This may be checked using~(\ref{doubler}).
Since the representation is built from 
$U_{\sqrt{v}}(\widehat{\mathfrak{gl}}(n))$ intertwining operators,
we have commuting actions of $\widetilde{H}_v$ and
$U_{\sqrt{v}}(\widehat{\mathfrak{gl}}(n))$.
\end{proof}

\begin{remark}
  \label{simpletwist}
  We could multiply $\Aa_{s_i}^{w\mathbf{z}}$ by any scalar value
  $\xi((w\mathbf{z})^{\alpha_i^\vee})$ where $\xi$ is a function
  of a complex parameter such that $\xi(z)\xi(z^{-1})=1$.
  Indeed, this modification affects neither the braid
  relation (\ref{abraida}) nor the scalar value of
  $\Aa^{s_i \mathbf{z}}_{s_i} \circ A_{s_i}^{\mathbf{z}}$.
\end{remark}


We wish to discuss the relationship between Theorem~\ref{heckeybeinstance} and
Proposition~\ref{JimboHeckeTheorem}.  There are two specializations
that shed light on this relationship. We shall consider the limits
$\mathbf{z}^{\alpha_i^\vee}\to 0$ (for all $i$) and $v\to 1$. We consider $\mathbf{z}^{\alpha_i^\vee}\to 0$,
or equivalently $x\to 0$, first.  

The $R$-matrices $R$ and $R(z)$ are related to the quantum groups
$U_q(\mathfrak{gl}(n))$ and its affinization $U_q(\widehat{\mathfrak{gl}}(n))$.  Recall that we defined
$R(x)=R-xR_{21}^{-1}$.
We recover $R$ by taking the limit $x\to 0$.
Let us move $\mathbf{z}$ in $\Treg(\C)$ 
in such a way that
$\mathbf{z}^{\alpha^\vee}\to 0$ for all $\alpha\in \Phi^\vee_+$.
Thus for every $w\in W$ and simple reflection $s_i$
\[(w\mathbf{z})^{\alpha_i^\vee}\longrightarrow
\left\{\begin{array}{ll}0&\text{if $w^{-1}\alpha_i^\vee\in\Phi^\vee_+$}\\
\infty&\text{if $w^{-1}\alpha_i^\vee\in\Phi^\vee_-$\,.}\end{array}\right.\]
The condition $w^{-1}\alpha_i^\vee\in\Phi^\vee_+$ is equivalent to $s_iw>w$.
It follows that
\[D_i(w\mathbf{z})
\longrightarrow
\left\{\begin{array}{ll}0&\text{if $s_iw>w$}\\
v-1&\text{if $s_iw<w$\,.}\end{array}\right.\]
We also have the following specialization of (\ref{baxexamplea}):
\[\Aa_{s_i}^{w\mathbf{z}}\longrightarrow\left\{
  \begin{array}{ll}\sqrt{v}(\tau R)_{i,i+1}&\text{if $s_iw>w$}\\
    \sqrt{v}(\tau{R})^{-1}_{i,i+1}&\text{if $s_iw<w$\,.}
    \end{array}
  \right.
\]

These specializations give a representation $\mathfrak{M}(0)$
of the finite Hecke algebra $\mathcal{H}_v$ which admits
the following description. If $U$ is any vector space,
we will denote by $W\cdot U$ the direct sum of $|W|$ copies
of $U$, indexed by $W$. Thus an element of $W\cdot U$ is
a tuple $(u_w|w\in W)$ with $u_w\in U$. The representation
$\mathfrak{M}(0)$ acts on $W\cdot\bigl(\otimes^rV)$, where
$V=\mathbb{C}^n$, and in this representation the above calculations show that
$T_i((\phi_{w}))=(\psi_{w})$ with
\begin{equation}
  \label{limitingrep}
\psi_{w} =
\left\{\begin{array}{ll}
(v-1)\phi_{w}+\sqrt{v}(\tau R)^{-1}_{i,i+1}\phi_{s_iw}
    &\text{if $s_iw>w$}\\
    \sqrt{v}(\tau R)_{i,i+1}\phi_{s_iw}&\text{if $s_iw<w$}.\end{array}\right.
\end{equation}
(We are writing $\phi_{w}$ instead of $\phi_{w\mathbf{z}}$
because $\mathbf{z}$ has been specialized to a limiting
value and no longer plays a role.)
This representation does not extend to a representation
of $\widetilde{\mathcal{H}}_v$, nor of $U_{\sqrt{v}}(\widehat{\mathfrak{gl}}(n))$.
However since $(\tau R)_{i,i+1}$ and $(\tau R)^{-1}_{i,i+1}$ are endomorphisms
of $\otimes^r V$ as a $U_{\sqrt{v}}(\mathfrak{gl}(n))$-module, we
do have a commuting action of $U_{\sqrt{v}}(\mathfrak{gl}(n))$.
  
\begin{remark}
The representation $\mathfrak{M}(0)$ is
closely related to the regular representation of
$\mathcal{H}_v$. Indeed, consider the action
$T_i(\phi_w)=(\psi_w)$ of $\mathcal{H}_v$ on $W\cdot\mathbb{C}$,
in which
\[\psi_w = \left\{
\begin{array}{ll}\sqrt{v}\phi_{s_iw}&\text{if
  $s_iw>w$}\\
  (v-1)\phi_w+\sqrt{v}\phi_{s_iw}&\text{if $s_iw<w$\,.}
\end{array}\right.
\]
This is just the regular representation
of $\mathcal{H}_v$. Compare this with $\mathfrak{M}(0)$, which 
differs from the regular representation by
involving the $R$-matrix. The regular representation is induced
from the trivial representation of $\mathbb{C}$, but the similar
representation $\mathfrak{M}(0)$ is not so induced.
\end{remark}

Our goal is to show that this representation generalizes that of Proposition~\ref{JimboHeckeTheorem}.
Let us first consider the limiting case where $v \rightarrow 1$. In this case
the Hecke algebra $\mathcal{H}_v$ degenerates to the symmetric group $S_r$. 
If $U$ is an $S_r$-module, we may define a
``wreathed'' $S_r$-module structure on $W\cdot U$ as follows. If $(\phi_w | w \in W)$ 
is an element of $W\cdot U$, and $y \in W$, define
$y (\phi_w) = (\psi_w)$ where $\psi_w = y \cdot \phi_{y^{- 1} w}$. This wreathed
representation $W\cdot U$
is isomorphic to the tensor product of $U$ with the regular representation
$\pi_{\text{reg}}$ of $S_r$. Now the regular representation has the
property that $U\otimes\pi_{\text{reg}}$ is isomorphic to the
direct sum of $\dim (U)$ copies of the regular representation of $S_r$, so it
might seem that we cannot recover the representation $U$ from it. However this
is not true since $W\cdot U$ comes equipped with an embedding of
$U \longrightarrow W\cdot U$, namely the diagonal map, and pulling back the
representation on $W\cdot W$ recovers the representation of~$U$.

Our next result is a Hecke algebra variant of this wreathed representation. When applied
to the representation described by (\ref{limitingrep}), we obtain the representation
of Proposition~\ref{JimboHeckeTheorem}. For part (iii), the Hecke algebra
$\mathcal{H}_v$ has an automorphism $\ast$ of order 2 which we now describe.
The algebra $\mathcal{H}_v$ has basis $T_w$ indexed by $w\in W$ such that
$T_{s_i}=T_i$, and $T_{ww'}=T_wT_{w'}$ when $\ell(ww')=\ell(w)+\ell(w')$.
Then $T_w^\ast=(-v)^{\ell(w)}T_{w^{-1}}^{-1}$. It is easy to prove that
$\ast$ is an automorphism since $T_i^\ast=-v T_i^{-1}$ satisfy the same quadratic
and braid relations as the~$T_i$.

\begin{theorem}
  \label{wreaththm}
  Let $U$ be an $\mathcal{H}_v$-module. 
  \begin{enumerate}
    \item[(i)]
    The vector space $W\cdot U$ has an
   $\mathcal{H}_v$-module structure in which $T_i$ acts by the operator
   $\mathfrak{T}_i$ on $W\cdot U$ defined by $\mathfrak{T}_i : (\phi_w) \longmapsto
   (\psi_w)$ where
   \[ \psi_w = \left\{ \begin{array}{ll}
       T_i \phi_{s_i w} & \text{if $s_i w < w$,}\\
       (v - 1) \phi_w + v T_i^{- 1} \phi_{s_i w} & \text{if $s_i w > w$} .
     \end{array} \right. \]
  \item[(ii)]
  Let $\Delta : U \longrightarrow W\cdot U$ be the diagonal map, so $\Delta \phi$
  is the tuple $(\phi_w)$ with $\phi_w = \phi$ for all $w$. Then $\Delta$ is
  an $\mathcal{H}_v$-module homomorphism.
  \item[(iii)] Let $\Delta^\ast:U\to W\cdot U$ be the map that sends $\phi\in U$ to
    $(\phi_w)\in W\cdot U$ where $\phi_w=(-v)^{\ell(w)}\phi$. Then
    $\Delta^\ast$ is a $\mathcal{H}_v$-module homomorphism, where we
    give $U$ the $\mathcal{H}_v$ module structure twisted by the
    involution $\ast$. For the generators, this means that
    \begin{equation}
      \label{twisteddiagonal}
      \mathfrak{T}_i\Delta^\ast(\phi)=\Delta^\ast(T_i^\ast\phi).
    \end{equation}
  \end{enumerate}
\end{theorem}

\begin{proof}
  We must show that the $\mathfrak{T}_i$ satisfy the quadratic and braid
  relations. If $y \in W$ and $\phi \in U$ let $\phi^y = (\phi_w)$ where
  $\phi_w = \phi$ if $w = y$, 0 otherwise. These elements span $W\cdot U$, and so
  it is sufficient to check the quadratic and braid relations on these. We
  have
  \[ \mathfrak{T}_i \phi^y = \left\{ \begin{array}{ll}
       (v - 1) \phi^y + (T_i \phi)^{s_i y} & \text{if $s_i y > y$,}\\
       v (T_i^{- 1} \phi)^{s_i y} & \text{if $s_i y < y$} .
     \end{array} \right. \]
  It is now straightforward to check that $\mathfrak{T}_i^2 \phi^y = (v - 1)
  \mathfrak{T}_i \phi^y + v \phi^y$, taking the cases $s_i y > y$ and $y > s_i
  y$ separately. This proves the quadratic relation.
  
  Now let us check the braid relations. We assume that $s_i s_j$ has order 3,
  leaving the case where $s_i$ and $s_j$ commute to the reader. We must show
  that $\mathfrak{T}_i \mathfrak{T}_j \mathfrak{T}_i \phi^w =\mathfrak{T}_j
  \mathfrak{T}_i \mathfrak{T}_j \phi^w$. Let $y$ be the element of maximal
  length in the coset $\langle s_i, s_j \rangle y$ of order $6$. First, if $w
  = y$, then both $\mathfrak{T}_i \mathfrak{T}_j \mathfrak{T}_i \phi^w$ and
  $\mathfrak{T}_j \mathfrak{T}_i \mathfrak{T}_j \phi^w$ equal $v^3 ((T_i T_j
  T_i)^{- 1} \phi)^{s_i s_j s_i}$, so this case is proved. To prove the other
  cases, we may use the quadratic relation. For example, we can prove that
  $\mathfrak{T}_i \mathfrak{T}_j \mathfrak{T}_i (\mathfrak{T}_i \phi^y)
  =\mathfrak{T}_j \mathfrak{T}_i \mathfrak{T}_j (\mathfrak{T}_i \phi^y)$ as
  follows. The left-hand side equals
  \[ (v - 1) \mathfrak{T}_i \mathfrak{T}_j \mathfrak{T}_i \phi^y +
     v\mathfrak{T}_i \mathfrak{T}_j \phi^y \]
  while using the already-proved fact that $\mathfrak{T}_i \mathfrak{T}_j
  \mathfrak{T}_i \phi^y =\mathfrak{T}_j \mathfrak{T}_i \mathfrak{T}_j \phi^y$,
  the right-hand side equals $\mathfrak{T}_j^2 \mathfrak{T}_i \mathfrak{T}_j
  \phi^y = (v - 1) \mathfrak{T}_j \mathfrak{T}_i \mathfrak{T}_j \phi^y + v \mathfrak{T}_i \mathfrak{T}_j 
  \phi^y$. Now again using $\mathfrak{T}_i \mathfrak{T}_j \mathfrak{T}_i
  \phi^y =\mathfrak{T}_j \mathfrak{T}_i \mathfrak{T}_j \phi^y$, we have proved
  $\mathfrak{T}_i \mathfrak{T}_j \mathfrak{T}_i (\mathfrak{T}_i \phi^y)
  =\mathfrak{T}_j \mathfrak{T}_i \mathfrak{T}_j (\mathfrak{T}_i \phi^y)$,
  which implies that $\mathfrak{T}_i \mathfrak{T}_j \mathfrak{T}_i \phi^{s_i
  y} =\mathfrak{T}_j \mathfrak{T}_i \mathfrak{T}_j \phi^{s_i y}$. Continuing
  in this fashion we get $\mathfrak{T}_i \mathfrak{T}_j \mathfrak{T}_i \phi^w
  =\mathfrak{T}_j \mathfrak{T}_i \mathfrak{T}_j \phi^w$ for all six elements
  of the coset $\langle s_i, s_j \rangle y$.
  
  Next we prove (ii). We must show that $\Delta
  T_i =\mathfrak{T}_i \Delta$. Apply both sides to $\phi \in U$. What we must
  show is that if $\phi_w = \phi$ for all $w \in W$ then $\mathfrak{T}_i
  (\phi_w) = (\psi_w)$ where $\psi_w = T_i \phi$ for all $w$. If $s_i w < w$
  then $\psi_w = T_i \phi$ by definition. On the other hand if $s_i w > w$,
  then $\psi_w = v T_i^{- 1} \phi + (v - 1) \phi$. Recalling that $T_i^{- 1} =
  v^{- 1} (T_i - (v - 1))$, we have $\psi_w = T_i \phi$ in this case also.

  For (iii), it is sufficient to check (\ref{twisteddiagonal}) since the
  $T_i$ generate $\mathcal{H}$. Since $\mathcal{H}_v$ is a semisimple algebra,
  we may assume that $\phi$ is an eigenvector for $T_i$. Because
  of the quadratic relation, the possible eigenvalues are $-1$ and $v$.
  Thus we may assume that either $T_i\phi=-\phi$ and $T_i^\ast\phi=v\phi$,
  or that $T_i\phi=v\phi$ and $T_i^\ast\phi=-\phi$. Then 
  (\ref{twisteddiagonal}) may be checked directly in both cases.
\end{proof}

The limiting representation (\ref{limitingrep}) is obtained from
the representation of Proposition~\ref{JimboHeckeTheorem} by applying the
construction in Theorem~\ref{wreaththm}. Note that because
$T_i=\sqrt{v}(\tau R)_{i,i+1}$ is $U_{\sqrt{v}}(\mathfrak{gl}(n))$
equivariant, the homomorphism $\Delta$ is a
$U_{\sqrt{v}}(\mathfrak{gl}(n))$-module homomorphism as well
as an $\mathcal{H}_v$-module homomorphism.

\begin{remark} 
Schur-Weyl-Jimbo duality implies the existence of a functor due to \cite{JimboHecke} from the category of finite dimensional right $\mathcal{H}_v$-modules to the category of $U_{\sqrt{v}}(\mathfrak{gl}(n))$-modules of level $r$ given by the following formula:
\[ \mathcal{J} : M \longmapsto M \otimes_{\mathcal{H}_v} (\otimes^r V). \]
This functor was generalized by Chari and Pressley \cite{ChariPressleyAffine} to a functor that takes finite dimensional (right) $\widetilde{\mathcal{H}}_v$-modules $M$ to finite dimensional $U_{\sqrt{v}}(\widehat{\mathfrak{gl}}(n))$-modules. More precisely, a module $M$ is mapped to a $U_{\sqrt{v}}(\mathfrak{gl}(n))$-module using Jimbo's functor,  which is then given a $U_{\sqrt{v}}(\widehat{\mathfrak{gl}}(n))$-module structure by an explicit prescription on the extra generators of $U_{\sqrt{v}}(\widehat{\mathfrak{gl}}(n))$ using the $\widetilde{\mathcal{H}}_v$-module structure on $M$. Theorem \ref{wreaththm} allows us to write a new functor from the category of finite dimensional right $\mathcal{H}_v$-modules to the category of $U_{\sqrt{v}}(\mathfrak{gl}(n))$-modules of:
\[ \mathcal{J}' : M \longmapsto M \otimes_{\mathcal{H}_v} (W \cdot \otimes^r V). \]
By Theorem \ref{wreaththm} (ii), this will be a generalization of Jimbo's functor. It would be interesting to study the properties of this new functor, and also to understand if we can endow $M \otimes_{\mathcal{H}_v} (W \cdot \otimes^r V)$ with a $U_{\sqrt{v}}(\widehat{\mathfrak{gl}}(n))$-module structure when $M$ is an $\widetilde{\mathcal{H}}_v$-module, therefore generalizing the functor due to Chari and Pressley. 
\end{remark}

\section{Drinfeld twists for quantum groups}\label{Drinfeldtwist}

In this section we introduce Drinfeld twists for quantum groups and $R$-matrices. This will allow us to generalize the setup in Section~\ref{HeckeRM} and match certain $R$-matrices with the intertwining operators acting on the twisted Jacquet modules, computed in~\cite{KazhdanPatterson}. 

In this section only, $q$ denotes an indeterminate or a non-zero complex number which is not to be confused with our earlier use of $q$ as the cardinality of the residue field of our local field and $\mathfrak{g}$ will denote either $\mathfrak{gl}(n)$ or $\widehat{\mathfrak{gl}}(n)$. The quantum group $U_q(\mathfrak{g})$
is a quasitriangular Hopf algebra that can be written down in terms of
generators and relations, see chapters 6 and 9 of Chari and Pressley
\cite{ChariPressleyBook}. In particular it is a Hopf algebra with
comultiplication $\Delta$ and antipode $S$, and with an invertible element 
$\mathcal{R} \in U_q(\mathfrak{g}) \widehat \otimes U_q(\mathfrak{g})$\footnote{Here $\widehat{\otimes}$ stands for a completion of the usual tensor product so $\mathcal{R}$ is an infinite sum of terms of the form $\mathcal{R}_i^{(1)} \otimes \mathcal{R}_i^{(2)}$, where $\mathcal{R}_i^{(1)}, \mathcal{R}_i^{(2)} \in U_q(\mathfrak{g})$. Still, for any tensor product of finite-dimensional modules, all but finitely many terms in the sum act as $0$ and therefore the action of $\mathcal{R}$ is well-defined.}, called a universal $R$-matrix,
that satisfies the following conditions:
\[ (\Delta \otimes \text{id}) \mathcal{R} = \mathcal{R}_{13} \mathcal{R}_{23}  \]
\[ (  \text{id} \otimes \Delta) \mathcal{R} = \mathcal{R}_{13} \mathcal{R}_{12}  \]  
\[ \Delta^{op} (x) = \mathcal{R} \Delta(x) \mathcal{R}^{-1}, \forall x \in U_q(\mathfrak{g}) \]
where $\Delta^{op} (x) = \tau \Delta (x)$ and $\tau (a \otimes b) = b \otimes a$. One can explicitly build
such an element using the Drinfeld quantum double construction; see \cite{Drinfeld}.

An important observation is that the $R$-matrices mentioned in Section \ref{HeckeRM} can be obtained from the action of the universal $R$-matrix of the corresponding quantum group. In particular, if V is the defining $U_q(\mathfrak{gl}_n)$-module, then the matrix $R$ defined in equation \eqref{glnRmatrix} is just the action of the universal $R$-matrix of $U_q(\mathfrak{gl}_n)$ on $V \otimes V$. Using this, it is then a standard exercise that $R$ satisfies the Yang-Baxter equation and $\tau R$ is an $U_q(\mathfrak{gl}_n)$-module homomorphism. 

Similarly, the $R$-matrix $R(xy^{-1})$ defined in \eqref{glnhatRmatrix} can be obtained from the action of the universal $R$-matrix of $U_q(\widehat{\mathfrak{gl}}(n))$ on the tensor product $V_x \otimes V_y$ of standard evaluation modules. 

We are interested in multiparameter deformations of these solutions and the quantum groups 
they are associated with. This was done by Drinfeld \cite{DrinfeldQuasiHopf} and Reshetikhin \cite{ReshetikhinMultiparameter}.
We note that similar results have been independently obtained by Artin, Schelter and Tate~\cite{ArtinSchelterTate} 
and Sudbery \cite{Sudbery}. 

Assume there is an  element $F = \sum_i f^i \otimes f_i \in U_q(\mathfrak{g}) \otimes U_q(\mathfrak{g})$ 
that satisfies the following conditions:
\begin{equation} \label{Fmatrix}
 (\Delta \otimes \text{id})F = F_{13} F_{23}, \,\,\, (\text{id} \otimes \Delta) F = F_{13} F_{12}, \,\,\,
F_{12} F_{13}F_{23} = F_{23} F_{13} F_{12}, \,\,\, F F_{21} = 1 
\end{equation}
where $F_{21} = \tau F = \sum_i f_i \otimes f^i$. Let $u =\sum_i f^i S(f_i)$. One can obtain a new quasitriangular Hopf algebra 
by $twisting$ by $F$ as follows:

\begin{theorem}[\cite{ReshetikhinMultiparameter}]
  \label{twistedhopfalgebra}
  There is a quasitriangular Hopf algebra $U^F_q(\mathfrak{g})$ that is 
$U_q(\mathfrak{g})$ as a vector space. It has the same unit, counit and multiplication as 
$U_q(\mathfrak{g})$. The comultiplication, antipode and universal $R$-matrix are given by 
\[\Delta^F(a) =F \Delta(a)F^{-1} , \,\,\, S^F(a) = u S(a) u^{-1}, \,\,\,  \mathcal{R}^F = F_{21} \mathcal{R} F^{-1}. 
\]
\end{theorem}

Consider $U_q(\mathfrak{gl}(n))$ as an algebra over $\mathbb{C}[[h]]$ generated by 
$H_i, X_i, Y_i$  subject to the relations presented in \cite{ReshetikhinMultiparameter}. The relation between $U_q(\mathfrak{g})$ defined in terms of $q$ and defined in terms of $h$ is rather subtle; it is discussed in chapter 9 of \cite{ChariPressleyBook}. For our purposes we only note that $q=e^h$. Set 
\[ F =  \exp \bigl(\sum_{i<j} (H_i \otimes H_j - H_j \otimes H_i) f_{ij}  \bigr)\]
where $f_{ij}$ are non-zero complex numbers. Then $F$ satisfies equation ($\ref{Fmatrix}$) and therefore 
there exists an associated quasitriangular Hopf algebra $U^F_q(\mathfrak{gl}(n))$. 

We note that the action of $U^F_q(\mathfrak{gl}(n))$ on the standard representation $V$ is the same as the action of $U_q(\mathfrak{gl}(n))$. However, the action of $U^F_q(\mathfrak{gl}(n))$ on $V \otimes V$ depends on the comultiplication which in turn depends on $F$. It might therefore be different than the action of $U_q(\mathfrak{gl}(n))$ on $V \otimes V$. 

The action of $\mathcal{R}^F$ on the standard module produces the following solution of the Yang-Baxter equation (\cite{ReshetikhinMultiparameter}):
\[ R^{\gamma} = \sum_i {q} \, e_{i i} \otimes e_{i i} + \sum_{i < j} \gamma^{-1}_{ij} e_{i i} \otimes
   e_{j j} + \sum_{i > j} \gamma_{ij} e_{i i} \otimes
   e_{j j} + \bigl({q} -q^{-1}\bigr) \sum_{i > j} e_{i j} \otimes e_{j i} \]
where $\gamma_{ij} = \exp (2f_{ij} - 2f_{i, j-1} - 2f_{i-1,j} +2f_{i-1,j-1})$.

Since $R^{\gamma}$ is obtained from the action of a universal $R$-matrix, it satisfies the Yang-Baxter equation. Moreover, $\tau R^{\gamma}$ is
a $U^F_q(\mathfrak{gl}(n))$-module homomorphism. 

\begin{proposition}
Proposition~\ref{JimboHeckeTheorem}
holds with $U_q(\mathfrak{gl}(n))$ replaced by $U^F_q(\mathfrak{gl}(n))$, where now the $T_i$'s
act by $q(\tau R^{\gamma})_{i,i+1}$.  
\end{proposition}
\begin{proof}
The braiding $q\tau R^{\mathbf{\gamma}}$ satisfies the quadratic relation:
\[ T^2 = (v - 1)T + v.\]
It also satisfies the braid relation because $R^{\gamma}$ is a solution of the Yang-Baxter equation. The statement follows  from these two facts, similarly to the proof of Proposition~\ref{JimboHeckeTheorem}.
\end{proof}

In the affine case, $U_q(\mathfrak{gl}(n))$ is a subalgebra of $U_q( \widehat{\mathfrak{gl}} (n))$ 
as mentioned above. It follows that $F \in U_q(\mathfrak{gl}(n)) \otimes U_q(\mathfrak{gl}(n)) \subset U_q(\widehat{\mathfrak{gl}}(n)) \otimes U_q(\widehat{\mathfrak{gl}}(n))$; 
denote by $U^F_q(\widehat{\mathfrak{gl}}(n))$ the quasitriangular Hopf algebra obtained from 
$U_q(\widehat{\mathfrak{gl}}(n))$ after twisting with $F$. 

For every standard $U_q(\widehat{\mathfrak{gl}}(n))$-module $V_x$, there is a standard $U^F_q(\widehat{\mathfrak{gl}}(n))$-module (which we will also denote by $V_x$) on which $U_q(\widehat{\mathfrak{gl}}(n))$ and $U^F_q(\widehat{\mathfrak{gl}}(n))$ act identically (remember that as algebras $U_q(\widehat{\mathfrak{gl}}(n))$ and $U^F_q(\widehat{\mathfrak{gl}}(n))$ are the same). The action of $\mathcal{R}^F$ on $V_x \otimes V_y$ 
can be computed explicitly because we know the action of $F$ on $V_x \otimes V_y$; it is 
the same as the action of $F$ on $V \otimes V$ where $V$ is the standard module of $U_q(\mathfrak{gl}(n))$. 
We obtain a solution of the parametrized Yang-Baxter equation that takes the following form:
\[ R^{\gamma}(x) = \sum_i (q-xq^{-1}) \, e_{i i} \otimes e_{i i} + \sum_{i < j} \gamma^{-1}_{ij}(1-x) e_{i i} \otimes
   e_{j j} + \sum_{i > j} \gamma_{ji} (1-x)e_{i i} \otimes
   e_{j j} \]
\[+ \bigl({q} -q^{-1}\bigr) \sum_{i > j} e_{i j} \otimes e_{j i} + x\bigl({q} -q^{-1}\bigr) \sum_{i < j} e_{i j} \otimes e_{j i} \]
with the property that $\tau R^{\gamma} (xy^{-1}): V_x \otimes V_y \to V_y \otimes V_x$ is 
a $U^F_q(\widehat{\mathfrak{gl}}(n))$-module homomorphism.

\begin{proposition}
\label{triangularityDrinfeld}
The $R$-matrix $R^{\gamma}(x)$ satisfies equation \textnormal{(}$\ref{doubler}$\textnormal{)}:
\[   \tau R^{\gamma}(x)\circ \tau R^{\gamma}(x^{-1})=\bigl(q-\frac{x}{q}\bigr)\bigl(q-\frac{1}{xq}\bigr).\]
 In the limit $x \to 0$, the following equation holds: $R^{\gamma}(0) = R^{\gamma}$.
\end{proposition}
\begin{proof}
The first fact follows from a computation of the left side. We will not show it here, but instead point the reader to the proof of Theorem 5 in \cite{BrubakerBuciumasBump}, where a similar computation is shown. The second fact follows from setting $x=0$ in the formula of $R^{\gamma}(x)$. 
\end{proof}

\begin{remark}
We note that for $\gamma_{ij} = 1$, the $R$-matrices $R^{\gamma}$ and $R^{\gamma}(x)$ become the usual $R$-matrix matrices $R$ and $R(x)$ in equations~\eqref{glnRmatrix} and~\eqref{glnhatRmatrix}.  
\end{remark}

We can now show that Theorem~\ref{heckeybeinstance} works in essentially the same way if we replace $U_q( \widehat{\mathfrak{gl}} (n))$ by $U^F_q( \widehat{\mathfrak{gl}} (n))$. Let $\mathcal{M}^F(w\mathbf{z})$ to be the $U^F_q( \widehat{\mathfrak{gl}} (n))$-module $V_{w(z_1)} \otimes \cdots \otimes V_{w(z_r)}$.
Define $\Aa_{s_i}^{w\mathbf{z}}: \mathcal{M}^F(w\mathbf{z})\to \mathcal{M}^F(s_iw\mathbf{z})$
to be the map
\begin{equation}
  \label{baxexampleatwist}
  \Aa_{s_i}^{w\mathbf{z}}=\frac{\sqrt{v}}{1-(w\mathbf{z})^{\alpha_i^\vee}}(\tau
  R^{\gamma}(w\mathbf{z}^{\alpha_i^\vee}))_{i,i+1}.
\end{equation}
where $\sqrt{v} = q$. As before, we define $\mathfrak{M}^F(\mathbf{z})=\bigoplus_{w\in W} \mathcal{M}^F(w\mathbf{z})$.
\begin{proposition}
  \label{heckeybeinstanceDrinfeld}
  With this data, the assumptions of
  Theorems~\ref{hecketheorem} and~\ref{hecketheoremaffine}
  are satisfied. The representation $\mathfrak{M}^F(\mathbf{z})$
  has commuting actions of $\widetilde{H}_v$ and
  $U_{\sqrt{v}}^F(\widehat{\mathfrak{gl}}(n))$.
\end{proposition}
\begin{proof}
The proof is identical to that of Theorem~\ref{heckeybeinstance}. The fact that $R^{\gamma}(x)$ satisfies the Yang-Baxter equation and that $\Aa_{s_i}^{w\mathbf{z}}$ is a $U_{\sqrt{v}}^F(\widehat{\mathfrak{gl}}(n))$-module homomorphism was explained above. The composition $\Aa^{s_i \mathbf{z}}_{s_i} \circ \Aa_{s_i}^{\mathbf{z}}$ equals the constant (\ref{doublesch}) due to Proposition~\ref{triangularityDrinfeld}.
\end{proof}
 
\section{Hecke modules from the metaplectic group}

In this section we will exhibit instances of our representation schema coming
from principal representations of the metaplectic group. This gives a
representation of Hecke algebra associated with the metaplectic L-group,
which we review.

Our first task will be to discuss metaplectic groups and
the metaplectic L-group. We will roughly follow the
discussion of McNamara~\cite{McNamaraPrincipal},
whose treatment of metaplectic groups is derived from
Matsumoto~\cite{Matsumoto} and Brylinski and Deligne
{\cite{BrylinskiDeligne}}.

As in Section 2, let $F$ be a nonarchimedean local field with ring of integers
$\mathfrak{o}$. We choose a local uniformizer $\varpi$ and let $q$ be residue
cardinality of $F$. Fix a positive integer $n$. We suppose that $q \equiv 1$
(mod $2 n$) so that $F$ contains the group $\mu_{2 n}$ of $2 n$-th roots of
unity. Let $(\;,\;)_{2n}$ denote the $2n$-th order Hilbert symbol for $F$.
If $m|2n$, let $(\;,\;)_m=(\;,\;)_{2n}^{2n/m}$. If $m=n$, we will denote
the $n$-th order Hilbert symbol $(\;,\;)_n$ as simply $(\;,\;)$.

Note that $\mu_n$ may be embedded both in $F^{\times}$ and
$\mathbb{C}^{\times}$. We will identify these. Thus there is an implicitly fixed 
embedding $\mathbf{j}$ of $\mu_n \subseteq F^{\times}$ into
$\mathbb{C}^{\times}$, which we will usually suppress from the notation.

Let $G (F)$ be the $F$-points of the split reductive group $G$, with $T (F)$ a
split maximal torus and $K^{\circ} = G (\mathfrak{o})$ a maximal compact
subgroup. As usual, $\Lambda$ will be the cocharacter group $X_*(T)$,
which may be identified with the character group $X^\ast(\widehat{T})$ of
the dual torus. Since \cite{McNamaraPrincipal} will be basic reference
in this section, we point out that $\Lambda$ is denoted $Y$ there,
and that what he denotes $\Lambda$ is our~$\Lambda^{(n)}$.

Let $\tilde{G}$ be an \textit{$n$-fold metaplectic cover of $G (F)$}. By
this we mean a nonalgebraic central extension
\begin{equation}
\label{metaone}
 1 \to \mu_{n} \to \tilde{G} \xrightarrow{p} G (F) \to 1
\end{equation}
or more generally an extension
\begin{equation}
\label{metatwo}
 1 \to \mu_m \to \tilde{G} \xrightarrow{p} G (F) \to 1
\end{equation}
where $n|m$,
such that the corresponding class $[\sigma]$ in $H^2(G(F),\mu_n)$ is
\textit{essentially of degree $n$} in the sense that $[\sigma^n]$ becomes
trivial in $H^2(G(F),\mathbb{C}^\times)$. As we will explain, such
a cover may be described in terms of a $W$-invariant symmetric bilinear
form $B : \Lambda \times \Lambda \longrightarrow \mathbb{Z}$ on $\Lambda$.

The relationship between $\widetilde{G}$ and the symmetric
bilinear form $B$ is then to be encoded in the following
fundamental commutation formula.
Let $\mu,\nu\in \Lambda$,
and let $a,b\in F^\times$. Then $a^\mu$ and $b^\nu$ will
represent the images of $a$ and $b$ under $\mu$ and $\nu$,
interpreted as cocharacters of $T$. Let $\widetilde{T}(F)$
be the preimage of $T(F)$ in $\widetilde{G}(F)$. It is
a 2-step nilpotent group. The commutation relation to be
satisfied is:
\begin{equation}
  \label{commutationbd} [x,y] = (a, b)^{B (\mu, \nu)}\;.
\end{equation}
where $x,y\in\widetilde{T}(F)$ satisfy $p(x)=a^\mu$ and $p(y)=b^\nu$.

We note that McNamara (like us) makes the simplifying assumption that the
ground field contains $\mu_{2 n}$. Also like us, he only requires the
form $B$ to be even on the lattice $\Lambda_{\text{coroot}}$ spanned
by $\Phi^\vee$. With these assumptions, starting from an extension
(\ref{metaone}) he proves the existence of a metaplectic cover satisfying
(\ref{commutationbd}).

In order to accomodate the following key example, we will
slightly modify his setup, instead starting from an extension
(\ref{metatwo}). In this example of (\ref{metatwo}) we have~$m=2n$.

\begin{example}
  \label{glnmetaex}
  Let $G=GL(r)$. We identify the
cocharacter group of the diagonal torus $T$ with $\mathbb{Z}^r$
in the usual way, so that if $\lambda=(\lambda_1,\cdots\lambda_r)$
then
\[t^\lambda = \left(\begin{array}{ccc}t^{\lambda_1}\\&\ddots\\&& t^{\lambda_r}\end{array}\right).\]
Let $\sigma$ be the cocycle in
$H^2(T(F),\mu_n)$ defined by
\[c\left(\left(\begin{array}{ccc}t_1\\&\ddots\\&& t_r\end{array}\right),
\left(\begin{array}{ccc}u_1\\&\ddots\\&& u_r\end{array}\right)\right)=
\left(\prod t_i,\prod u_i\right)_{2n}\prod_{i>j}(t_i,u_j)^{-1}.\]
\end{example}

\begin{proposition}
This cocycle extends to a class in $H^2(G(F),\mu_{2n})$ that is
essentially of degree $n$. If $B$ is the usual dot product on
$X_*(T)=\mathbb{Z}^r$,
then (\ref{commutationbd}) is satisfied.
\end{proposition}

\begin{proof}
The cocycle is a modification of the cocycles employed by Kazhdan and
Patterson~\cite{KazhdanPatterson} Section~0.1, so as in their paper it may be
deduced from Matsumoto's theorem for $\SL(n+1)$ that it can be extended to a
cocycle in $H^2(G(F),\mu_{2n})$. Raising it to the $n$-th power gives
the cocycle
\[(g,h)\mapsto(\det(g),\det(h))_2.\]
The class of this splits in $H^2(G(F),\mu_8)$. Indeed,
Weil~\cite{WeilCertains} associated an $8$-th root of unity $\gamma(Q)$
with every quadratic form over a local field $F$. If 
$a\in F^\times$, then we apply this to the quadratic form
$x^2-a y^2$ on $F^2$. Defining $\gamma_0(a)$ to be $\gamma(x^2-a y^2)$ it
follows from the last formula on page 176 of~\cite{WeilCertains}
that
\[\gamma_0(a)\gamma_0(b)/\gamma_0(ab)=(a,b)_2.\]
Therefore the section $g\mapsto\gamma_0(\det(g))$ splits the
cocycle $[\sigma^n]$ in $H^2(G,\mu_8)$ or
$H^2(G,\mathbb{C}^\times)$. The fact that $B$ coincides
with the usual dot product is a simple computation.
\end{proof}

We return to the general case. Our setup is slightly
different than McNamara's and so we will prove a couple
of foundational lemmas. 

We will \textit{assume} that we can find a symmetric bilinear form $B$ on
$\Lambda$ such that (\ref{commutationbd}) is satisfied, but for this to be
reasonable we would like to know that the element of $\mu_{m}$ on the left-hand
side is actually an $n$-th root of unity. To check this, let us choose a
section $\mathbf{s}:G(F)\to\widetilde{G}(F)$.

\begin{lemma}
  Suppose that $\widetilde{G}$ is an extension (\ref{metatwo})
  whose class $[\sigma]\in H^2(G,\mu_{m})$ with $n|m$ is represented by a
  cocycle $\sigma$. Assume that the class of $\sigma^n$ is trivial
  in $H^2(G,\C^\times)$. Then $[x,y]\in\mu_n$ for
  $x,y\in\widetilde{T}(F)$.
\end{lemma}

\begin{proof}
  It is enough to check this for $x = \mathbf{s} (t)$ and $y = \mathbf{s} (u)$.
  We may also assume that the cocycle
  \begin{equation}
    \label{sgcocycle} \sigma (g, h) = \mathbf{s} (g)\, \mathbf{s} (h) \, \mathbf{s}
    (gh)^{- 1} .
  \end{equation}
  Let us assume that $\mathbf{s} (1) = 1$ so $\sigma (1, 1) = 1$. Our assumption
  implies that there is a map $a : G (F) \longrightarrow \C^{\times}$
  such that
  \begin{equation}
    \label{splita} \sigma (g, h)^n = \frac{a (gh)}{a (g) a (h)} \hspace{0.27em} .
  \end{equation}
  Since $\sigma (1, 1) = 1$ we have $a (1) = 1$. Since $tu=ut$, equation
  (\ref{sgcocycle}) implies that
  \begin{equation}
    \label{xyccy} [x, y] = \sigma (t, u) \sigma (u, t)^{- 1} .
  \end{equation}
  Raising this to the $n$-power and substituting (\ref{splita}) everything
  cancels, so $[x, y]$ is an $n$-th root of unity.
\end{proof}

Brylinski and Deligne require the quadratic form $B$ to be even, with
associated integer-valued quadratic form $Q (\mu) = B (\mu,\mu) / 2$ on the
weight lattice. Due to our modified setup, we (like McNamara) only claim the
restriction of $B$ to the lattice $\Lambda_{\text{coroot}}$ generated by
$\Phi^{\vee}$ to be even. So for us $Q$ is defined on
$\Lambda_{\text{coroot}}$. An integer-valued symmetric bilinear
form is even on a lattice if it is even on a set of generators,
so for this the following Lemma is sufficient.

\begin{lemma}
  Suppose that $\widetilde{G}$ is an extension (\ref{metatwo})
  whose class in $H^2(G,\mu_{m})$ is essentially of degree~$n$.
  Let $\alpha^{\vee} \in \Phi^{\vee}$. Then there is an even integer $2 k$
  such that for $a, b \in F^{\times}$ if $p (x) = a^{\alpha^{\vee}}$ and $p
  (y) = b^{\alpha^{\vee}}$ then
  \begin{equation}
    \label{corooteven} [x, y] = (a, b)^{2 k} .
  \end{equation}
\end{lemma}

Thus $Q(\alpha^\vee)=k$, and this Lemma justifies the assertion
that $Q(\alpha^\vee)=B(\alpha^\vee,\alpha^\vee)/2\in\Z$.

\begin{proof}
  Corresponding to the root $\alpha^{\vee}$ is an embedding $i_{\alpha^{\vee}}
  : \SL_2 \longrightarrow G$. The commutator $[x, y]$ is independent of
  the choice of section, and hence of the choice of cocycle representing the
  cover. Applying Theorem~9.2 of Moore~{\cite{MooreExtensions}} shows that the
  restriction of the cocycle $\sigma$ defining the cover to $\SL_2$ along
  this embedding can be chosen to be the composition of the universal cocycle
  $b$ in that theorem with a homomorphism $\pi_1 (G) \longrightarrow \mu_{m}$.
  
  With this choice of cocycle, we pull the cocycle back to $F^{\times}$ via
  the homomorphism
  \[ y \longmapsto i_{\alpha^{\vee}} \left(\begin{array}{cc}
       y & \\
       & y^{- 1}
     \end{array}\right) . \]
  Then {\cite{MooreExtensions}} Theorem~9.2 shows that this cocycle in $Z^2
  (F^{\times}, \mu_{m})$ is a Steinberg cocycle in the sense of
  {\cite{MooreExtensions}} Definition 3.1. Therefore by Moore's Theorem 3.1
  with $A = \mu_n$, we have $\sigma (a^{\alpha^{\vee}}, b^{\alpha^{\vee}}) = (a,
  b)^{k'}_N$ for some integer $k'$, where $\left( \;, \; \right)_N$ is the
  $N$-th order Hilbert symbol, and $N$ is the number of roots of unity in $F$.
  We have assumed that the class $[\sigma^n]$ is trivial in $H^2 (G,
  \C^{\times})$, and after adjoining enough roots of unity to $F$ we
  may assume that it is trivial in $H^2 (G, \mu_N)$. Pulling this back along
  $i_{\alpha^{\vee}}$ we see that $[\sigma^n]$ is trivial in $H^2 (\SL_2, \mu_N)$.
  By \cite{MooreExtensions} Theorem~10.4 this implies that $k'$ divides $N /
  n$. Thus $\sigma (a^{\alpha^{\vee}}, b^{\alpha^{\vee}}) = (a, b)^k$ for some
  integer~$k$. Now (\ref{xyccy}) implies (\ref{corooteven}).
\end{proof}

The group $K^\circ$ splits in $\tilde{G}$ and we fix a splitting.
Denote by $H := C_{\tilde{T}}(\tilde{T} (\mathfrak{o}))$ the centralizer of
$\tilde{T} (\mathfrak{o}) =\tilde{T} \cap K^{\circ}$ in $\tilde{T}$. Choose a
basis $e_1, \ldots, e_r$ of $\Lambda$, and define constants $b_{i j}$ by
$b_{i j} = B (e_i, e_j)$. Every element
of $T$ may be uniquely written as $\prod t_i^{e_i}$ where $t_i \in
F^{\times}$. It may be deduced from (\ref{commutationbd}) that $H$ is
a maximal abelian subgroup of $\tilde{T}$ and it is characterized as follows.
An element of $\tilde{T}$ whose projection in $T (F)$ is $\prod t_i^{e_i}$
is in $H$ if and only if
$\operatorname{ord}_{\varpi} (\prod_j t_j^{b_{i, j}}) \equiv 0$ mod $n$
for $i = 1, \ldots, r$.
(This may be proved along the lines of Lemma~\ref{lgroupmeaning}
below, or see Lemma~1 of {\cite{McNamaraPrincipal}} for a proof.)

The lattice $\Lambda=X_*(T)$ is canonically isomorphic to $\tilde{T} / \mu_n
\tilde{T} (\mathfrak{o})$; we associate $\mu\in\Lambda$ with the image
$\varpi^\mu$ of a prime element $\varpi$ under the cocharacter map
$\mu:F^\times\to T(F)$. Let
\[ \Lambda^{(n)} = \{\mu \in \Lambda \hspace{0.17em} | \hspace{0.17em} B (y,
   \mu) \equiv 0 \hspace{0.17em} (n) \hspace{0.17em} \text{for all $y \in \Lambda$}
   \} . \]
Lemma~1 of {\cite{McNamaraPrincipal}} implies that the isomorphism of
$\Lambda$ with $\tilde{T} / \mu_n  \tilde{T} (\mathfrak{o})$ induces an
isomorphism of $\Lambda^{(n)}$ with $H / \mu_n  \tilde{T} (\mathfrak{o})$.

If $S$ is a complex torus, let $X^{\ast} (S)$ be its group of rational
characters. Then $S \mapsto X^{\ast} (S)$ is a contravariant functor that is
an equivalence of categories, from the category of complex tori to the
category of lattices. Thus the embedding $\Lambda^{(n)} \longrightarrow
\Lambda$ corresponds to a surjection $\hat{T} (\mathbb{C}) \longrightarrow
\hat{T}' (\mathbb{C})$, where $\hat{T}'$ is another complex torus, with
$\Lambda^{(n)}$ identified with $X^\ast(T')$. We will denote by $\mathbf{z}
\mapsto \mathbf{z}'$ this surjection.

We turn now to the definition of the metaplectic L-group,
from the point of view of~\cite{McNamaraPrincipal}. It is
sufficient to present root data. The root lattice is
$\Lambda^{(n)}$, which contains a root system that we now
describe. If $\alpha^{\vee} \in \Phi^{\vee} \subseteq \Lambda$, let $\alpha' =
n_{\alpha} \alpha^{\vee} \in \Lambda^{(n)}$, where
\[ n_{\alpha} = \frac{n}{\gcd (n, Q (\alpha^\vee))} . \]
The fact that $\alpha' \in \Lambda^{(n)}$ is proved in Theorem~12 of
{\cite{McNamaraPrincipal}}, where it is further shown that $\Phi' = \{ \alpha' |
\alpha^{\vee} \in \Phi^{\vee} \}$ is the root lattice of a
reductive Lie group, which we will denote $\hat{G}' (\mathbb{C})$. This
has the following significance for us: for metaplectic groups, we may apply the
representation schema with $\hat{T}' (\mathbb{C})$ and $\Phi'$ replacing
$\hat{T} (\mathbb{C})$ and $\Phi^{\vee}$, respectively. Thus the torus
$\hat{T}' (\mathbb{C})$ is the maximal torus in a
connected component of a complex reductive group with
root system $\{\alpha'\}$. This is the metaplectic L-group.

We note the formula, for $\mathbf{z}\in\widehat{T}$:
\begin{equation}
  \label{metacoroot}
  (\mathbf{z}')^{\alpha'}=\mathbf{z}^{n_\alpha\alpha^\vee}.
\end{equation}

Coset representatives of $\tilde{T} / H$ can be written as the image under
$\mathbf{s}$ of $\varpi^{\lambda}$, with $\lambda \in \Lambda$ and where
$\mathbf{s}: G \to \tilde{G}$ denotes the standard section. We will abuse notation and
denote the image $\mathbf{s} (\varpi^{\lambda})$ by $\varpi^{\lambda}$.

The construction of the principal series representation of $\tilde{G}$
resembles the construction in the non-metaplectic case. Start with an
unramified character of $H$, namely a genuine character $\chi$ on $H$ that is
trivial on $\tilde{T} \cap K^{\circ}$. Such characters are indexed by
Langlands parameters $\mathbf{z} \in \hat{T} (\mathbb{C})$, the dual torus. We
will write $\chi = \chi_{\mathbf{z}}$ when we wish to emphasize that
dependence. Let $i (\chi):= \Ind^{\tilde{T}}_H (\chi)$. The principal
series representation is obtained by parabolic induction: we inflate $i
(\chi)$ from $\tilde{T}$ to $\tilde{B} = p^{- 1} (B)$ and then induce to
$\tilde{G}$ to obtain $I (\chi) := \Ind^{\tilde{G}}_{\tilde{B}} (i
(\chi))$. Let $\delta$ denote the modular quasicharacter. The module
$I (\chi)$ consists of locally constant functions $f :\tilde{G} \to i (\chi)$ that satisfy
\[ f (bg) = \delta^{\frac{1}{2}} i (\chi) (b) f (g) \]
for all $b \in \tilde{B}, g \in \tilde{G}$, where $\tilde{G}$ acts by right
translation.

\begin{lemma}
  \label{lgroupmeaning}
  Let $\mathbf{z}_1, \mathbf{z}_2 \in \hat{T} (\mathbb{C})$. Then $i
  (\chi_{\mathbf{z}_1}) \cong i (\chi_{\mathbf{z}_2})$ if and only if
  $\mathbf{z}_1' =\mathbf{z}_2'$ in $\hat{T}' (\mathbb{C})$.
\end{lemma}

\begin{proof}
  We identify $T (F) = (F^{\times})^r$ as follows. Recall that we have chosen
  a basis $e_1, \cdots, e_r$ of $\Lambda = X_{\ast} (T)$. Every $s \in T (F)$ may be
  written uniquely as $\prod s_i^{e_i}$ with $s_i \in F^{\times}$, and we
  associate this with $s = (s_1, \cdots, s_r) \in (F^{\times})^r$. We also
  denote $\operatorname{ord} (s) = (\sigma_1, \cdots, \sigma_r) \in \mathbb{Z}^r$
  where $\sigma_i = \operatorname{ord} (s_i)$.

  Let $\tilde{s}\in\widetilde{T}(F)$, and let $s$ denote its image in $T(F)$.
  Let $\sigma=(\sigma_1,\cdots,\sigma_r)=\sigma(s)$. We will show that
  $\tilde{s}\in Z(\tilde{T}(F))$ if and only if $\sum_i b_{i j} \sigma_i \in
  n\mathbb{Z}$ for every $j$. First, suppose that $\tilde{s}$ is central. 
  Let $a\in F^\times$ and let $y\in \widetilde{T}(F)$ be such that $p(y)=a^{e_j}$.
  We have 
  $[\tilde{s}, y] = 1$ so by the commutation relation (\ref{commutationbd}), we see
  that
  \[ 1 = \prod_i (s_i, a)^{b_{i j}} = \left( \prod_i s_i^{b_{i j}}, a \right)
  \]
  for all $a \in F^\times$, and for all $j$. Using the fact that the kernel of the
  Hilbert symbol is $(F^{\times})^n$, this means that $\prod_i s_i^{b_{i j}}$
  is an $n$-th power, and consequently $\sum_i b_{i j} \sigma_i \equiv 0$
  modulo $n$. Conversely, if $\sum_i b_{i j} \sigma_i \equiv 0$ modulo $n$
  for all $j$, then we may choose $s_i = \varpi^{\sigma_i}$ and the same calculation shows
  that $s = (s_1, \cdots, s_r)$ is central.
  
  By Corollary 1 in Section 13.5 of {\cite{McNamaraPrincipal}}, a necessary
  and sufficient condition for $i (\chi_{\mathbf{z}_1}) \cong i
  (\chi_{\mathbf{z}_2})$ is that $\chi_{\mathbf{z}_1}$ and
  $\chi_{\mathbf{z}_2}$ have the same restriction to $Z (\tilde{T} (F))$. In
  view of this, it is enough to show that $\chi_{\mathbf{z}}$ is trivial on
  $Z (\tilde{T} (F))$ if and only if $\mathbf{z}' = 1$ in $\hat{T}'
  (\mathbb{C})$, since applying this to $\mathbf{z}=\mathbf{z}_1
  \mathbf{z}_2^{- 1}$ will then imply the theorem. In view of our definition
  of $\hat{T}' (\mathbb{C})$, we have $\mathbf{z}' = 1$ if and only if
  $\mathbf{z}^{\mu} = 1$ for all $\mu \in \Lambda^{(n)}$.
  
  We recall that $\Lambda = X^{\ast} (\hat{T})$ is identified with $\Lambda =
  X_{\ast} (T)$. Denote $z_i =\mathbf{z}^{e_i} \in \mathbb{C}^{\times}$. If
  $\mu \in \Lambda$ we may expand $\mu = \sum \sigma_i e_i$, and $\mu \in
  \Lambda^{(n)}$ if and only if $\sum_i b_{i j} \sigma_i \equiv 0$ mod $n$ for
  all $j$. Since $\mathbf{z}^{\mu} = \prod z_i^{\sigma_i}$, the condition
  for $\mathbf{z}' = 1$ is that $\prod z_i^{\sigma_i} = 1$ whenever $\sum_i
  b_{i j} \sigma_i \equiv 0$. On the other hand, if $\sigma = \operatorname{ord} (s)$
  as before, then $\chi_{\mathbf{z}} (s) = \prod z_i^{\sigma_i}$, so this is
  the same as the condition for $\chi_{\mathbf{z}}$ to be trivial on the $Z
  (\tilde{T} (F))$, and the statement is proved.
\end{proof}

\begin{remark}
\label{weylchoice}
Representatives for the Weyl group will
appear in formulas such as (\ref{intertwinerwithcs}) below. We will always choose
the representative of $w\in W=N(T)/T$ to lie in $N(T)\cap K^\circ$, which is lifted
to the metaplectic group via the chosen splitting over $K^\circ$. By abuse of
notation, we will denote this representative also as $w$. This is harmless
since we are inducing an unramfied representation, so
(\ref{intertwinerwithcs}) does not depend on the choice of representative modulo
$T(F)\cap K^\circ=T(\mathfrak{o})$.
\end{remark}

The subgroup $U(F)$ embeds canonically in $\tilde G$ by the trivial section.
Given a positive root $\alpha \in \Phi^+$ we denote by $U_{\alpha}$ the one
parameter unipotent subgroup corresponding to the embedding $i_{\alpha}: SL_2
\to G$. The intertwining operator $\Ai_w=\Ai_w^{\mathbf{z}}: I(\chi_{\mathbf{z}}) \to I(\chi_{w\mathbf{z}})$ can now be defined by 
\begin{equation}
\label{intertwinerwithcs}
\Ai_w^{\mathbf{z}} (f)(g) := \int_{U(F)\cap wU^-(F)w^{-1}} f(w^{-1}ug)\,du 
\end{equation}
when the integral is absolutely convergent, and by analytic continuation
otherwise. The composition $\Ai_{s_i}^{\mathbf{z}}\Ai_{s_i}^{\mathbf{z}}$
is a scalar, indeed
\begin{equation}
\label{metascalar}
\Ai_{s_i}^{s_i\mathbf{z}}\Ai_{s_i}^\mathbf{z}=
\frac{(1-v\mathbf{z}^{n_{\alpha_i}\alpha_i^\vee})(1-v\mathbf{z}^{-n_{\alpha_i}\alpha_i^\vee})}
     {(1-\mathbf{z}^{n_{\alpha_i}\alpha_i^\vee})(1-\mathbf{z}^{-n_{\alpha_i}\alpha_i^\vee})}=
     \frac{(1-v(\mathbf{z'})^{\alpha_i'})(1-v(\mathbf{z'})^{-\alpha_i'})}
          {(1-(\mathbf{z'})^{\alpha_i'})(1-(\mathbf{z'})^{-\alpha_i'})}\;.
\end{equation}
This follows from Proposition~4.4 in~\cite{McNamaraCS}.

We remind the reader of the point we made in the introduction, that it
is convenient to think of the Hecke algebra of a reductive group $G$
as something associated with the Langlands dual group $\widehat{G}(\C)$.

\begin{example}
\label{metalmtheory}
We may now give a basic example of the representation schema, making
use of the metaplectic dual-group. In the Assumptions of Section~\ref{schemasect},
we will take the torus to be not $\widehat{T}(\C)$ but $\widehat{T}'(\C)$,
and the root system there will be $\Phi'$. Let $\mathbf{z}'\in
\widehat{T}'(\mathbf{z})$. Define $\mathcal{M}(\mathbf{z}')$ to
be the principal series representation $I(\chi_{\mathbf{z}})$,
where $\mathbf{z}\in\widehat{T}(\C)$ maps to $\mathbf{z}'$ in
$\widehat{T}'(\C)$. This representation is well-defined by
Lemma~\ref{lgroupmeaning}. The assumptions are now satisfied
and so we obtain a representation of the affine Hecke algebra associated
with the metaplectic dual-group on
\[\mathfrak{M}(\mathbf{z}')=\bigoplus_{w\in W}\mathcal{M}(w \mathbf{z}').\]
In particular, the scalar (\ref{doublesch}) is to be interpreted
as
\[
  \frac{(1 - v(\mathbf{z'})^{\alpha_i'})(1 - v(\mathbf{z'})^{-\alpha_i'})}
  {(1 -(\mathbf{z'})^{\alpha_i'})(1 - (\mathbf{z'})^{-\alpha_i'})}
\]
and for this we may use (\ref{metascalar}) and (\ref{metacoroot}).
\end{example}

As before, instead of taking all of $I(\chi_{\mathbf{z}})$,
we may define $\mathcal{M}(\mathbf{z}')$ to be the module of
Whittaker coinvariants, or Iwahori fixed vectors, etc. We
will consider the Whittaker coinvariants in the next section.

\section{Metaplectic Whittaker functionals and $R$-matrices\label{whitandr}}

We will now take Whittaker coinvariants in the representations
of Example~\ref{metalmtheory}. Then
results of \cite{KazhdanPatterson,ChintaOffen,McNamaraCS}
can be reinterpreted as showing that the scattering matrix
of the intertwining operators on the Whittaker coinvariants
agrees with the R-matrices coming from quantum groups.

Results of \cite{KazhdanPatterson,ChintaOffen,McNamaraCS} allow us to make
the Whittaker representations more explicit. We review this, focussing
soon on the $\GL(r)$ case in this section. We will see that the same
representations appear also from $R$-matrices of quantum groups.

Just as in the linear group case,
given a generic character $\psi$ of $U(F)$, viewed as a subgroup of $\tilde{G}$, a Whittaker functional on a representation $V$ of $\tilde G$ is a functional $W$ such that 
\[ W(u \cdot \phi) = \psi(u)W(\phi) \quad\text{for all $u \in U(F), \phi \in V$}. \]
The vector space of Whittaker functionals on $I(\chi_{\mathbf{z}})$ has dimension equal to the cardinality of $\tilde T/H$, and a natural basis ${W}^{\chi}_b$, indexed by coset representatives $b$ for $\tilde T/H$. If $W^{\chi}$ denotes the $i(\chi)$-valued Whittaker functional on functions $\phi \in I(\chi)$ given by
\begin{equation} W^{\chi}(\phi) = \int_{U^-} \phi(u) \psi(u) \, du, \label{whitatident} \end{equation}
then a basis for the space of $\mathbb{C}$-valued Whittaker models arises by composing with a choice of basis for the space of functionals on $i(\chi)$. Because we wish to discuss the image of $K^\circ$-fixed vectors under the Whittaker model, it is convenient to choose a basis using the unique-up-to-scalar $K^\circ$-fixed (i.e. ``spherical'') vector $\phi^\circ$ in the representation $I(\chi)$. Indeed the map $\phi^\circ \mapsto \phi^\circ(1) =: v_0$ induces an isomorphism $I(\chi)^{K^\circ} \simeq i(\chi)^{K^\circ \cap \tilde{T}}$, and acting by the induced representation $\pi_\chi$ on $i(\chi)$, the set $\{ \pi_\chi(a) v_0 \}$ with $a$ running over a set of representatives for $\tilde{T} / H$ is a basis for $i(\chi)$. Thus, we have a dual basis $\{ \mathscr{L}_b \}$ for $i(\chi)^\ast$ indexed by elements $b \in \tilde{T} / H$ such that $\mathscr{L}_b( \pi_\chi(a) v_0 ) = \delta_{a,b}$. We select coset representatives $b \in \tilde{T} / H$ of the form $\varpi^\mu$ with $\mu \in \Lambda$ and the Whittaker functionals $W_b^\chi$ associated to $\chi := \chi_\mathbf{z}$ are then defined by
\begin{equation} 
  \label{normalizedfun}
  W_b^\chi := \mathbf{z}^{\mu} \cdot \mathscr{L}_b \circ W^\chi, \quad \text{if $b = \varpi^\mu$ with $\mu \in
    \Lambda$,} 
\end{equation}
where $W^{\chi}$ is the $i(\chi)$-valued Whittaker functional in (\ref{whitatident}). Taking the spherical vector $\phi^\circ \in I(\chi_{\mathbf{z}})$, the span of the functions $W_b^\chi (\phi^\circ)$ with $b \in \tilde{T} / H$ will be denoted $\mathcal{W}_{\mathbf{z}}$, and referred to as the space of spherical Whittaker functions for $I(\chi_{\mathbf{z}})$.

Let 
\[\bar{\Ai_w} := c_w^{(n)}(\chi)^{-1} \Ai_w \]
denote the normalized intertwiner, where $c_{w}^{(n)}(\chi)$ is defined on simple reflections $s=s_{\alpha}$ as 
\begin{equation} c_{s}^{(n)}(\chi_\mathbf{z}) = c_{s}^{(n)}(\mathbf{z}) =
  \frac{1-q^{-1}\mathbf{z}^{n_\alpha \alpha^\vee}}{1-\mathbf{z}^{n_\alpha
      \alpha^\vee}}=
  \frac{1-q^{-1}(\mathbf{z}')^{\alpha'}}{1-(\mathbf{z}')^{\alpha'}}.
\label{nalphafirstocc} \end{equation}
Then $c_w^{(n)}(\chi)$ is determined by the fact that, for all $w\in W$ such that $\ell(sw)=\ell(w)+1$, we have $c_{sw}^{(n)}(\chi) = c_s^{(n)}(\chi^w)c_w^{(n)}(\chi)$.
The advantage of the normalized intertwining integral is that they are
perfectly multiplicative, that is,
$\bar{\Ai}_w\bar{\Ai}_{w'}=\bar{\Ai}_{ww'}$,
whereas for the unnormalized intertwining integral, this is only true under
the assumption that $\ell(ww')=\ell(w)+\ell(w')$. The disadvantage of the
normalized intertwining integrals is that they have poles where
$c_w^{(n)}(\chi)$ vanishes. We will use them since we need results of
\cite{McNamaraCS,BrubakerBuciumasBump} that are expressed in terms of them.

Given a Whittaker functional $W$ on $I(^w \chi)$, the composition $W \circ
\bar{\Ai_{w}}$ is a Whittaker functional on $I(\chi)$. 
Thus
\begin{equation}\label{eq:KP}
W_a^{\chi^w} \circ \bar{\Ai_w} = \sum_{b \in \tilde T/H} \tau_{a,b} W^{\chi}_b
\end{equation}
with coefficients $\tau_{a,b} := \tau_{a,b}^{(w)}(\mathbf{z}) \in \mathcal{O} ({\Treg}(\C))$. 
To understand (\ref{eq:KP}), it is enough to determine the values of $\tau_{a,b}$ for simple reflections. According to Lemma I.3.3 in Kazhdan and Patterson \cite{KazhdanPatterson} or Theorem 13.1 in McNamara \cite{McNamaraCS}, in the case $w=s_{\alpha}$ (or just $s$ for short), the computation of the coefficients amounts to a rank one calculation which can be done uniformly for covers of split groups. Writing $\tau_{a,b}$ as
$\tau_{\nu, \mu}$ for $a = \pi^\nu$ and $b = \pi^\mu$, then we break $\tau_{\nu, \mu}$ into two pieces
 $$ \tau_{\nu,\mu} = \tau^1_{\nu,\mu} + \tau^2_{\nu,\mu}. $$
Here $\tau^1$ vanishes unless $\nu \sim \mu$ mod $\Lambda^{(n)}$ and $\tau^2$ vanishes unless $\nu \sim s(\mu) + \alpha^\vee$ mod $\Lambda^{(n)}$. Moreover:
\begin{equation} \tau^1_{\mu, \mu} = (1 - q^{-1}) \frac{\mathbf{z}^{\left(n_\alpha \lceil \frac{B(\alpha^\vee, \mu)}{n_\alpha Q(\alpha^\vee)} \rceil - \frac{B(\alpha^\vee, \mu)}{Q(\alpha^\vee)}\right) \alpha^\vee}}{1-q^{-1} \mathbf{z}^{n_\alpha \alpha^\vee}} \label{tauone} \end{equation}
where $\lceil x \rceil$ denotes the smallest integer at least $x$ and $n_\alpha$ is as in (\ref{nalphafirstocc}), and
\begin{equation} \tau^2_{s(\mu)+\alpha^\vee, \mu} = q^{-1} g(B(\alpha^\vee, \mu) - Q(\alpha^\vee)) \mathbf{z}^{-\alpha^\vee}\frac{1-\mathbf{z}^{n_\alpha \alpha^\vee}}{1-q^{-1} \mathbf{z}^{n_\alpha \alpha^\vee}}, \label{tautwo} \end{equation}
where the $n$-th order Gauss sum $g$ is normalized as in Section 13 of~\cite{McNamaraCS}.
These explicit formulas will be used in the next section. Note the form above for the two $\tau$ contributions is different than that of \cite{McNamaraCS} owing to our normalization of the Whittaker functionals in~(\ref{normalizedfun}).

In the remainder of this section, we restrict ourselves to the case of
$G = \textrm{GL}_r$ and the extension described in Example~\ref{glnmetaex}.
In this case, we may take $e_i$ to be the standard
basis of $X_*(T)\cong\mathbb{Z}^r$, where $T$ is the diagonal torus, and
\[B(e_i,e_j)=\left\{\begin{array}{ll}1&\text{if $i=j$,}\\
0&\text{if $i\neq j$.}\end{array}
\right.\]
The lattice $\Lambda^{(n)}=n \Lambda$. For the
definition of the functionals $W_b$, see Sections~7 and 8 of
\cite{BrubakerBuciumasBump}.  A basis consists of ${W}_b$ where $b = \varpi
^{\nu}$ with $\nu \in \Delta$ of the form:
\[ \nu - \rho = \sum_i c_i e_i \]
where $0\leqslant c_i \leqslant n-1$ and $e_i$ is the weight $(0, \cdots 1, \cdots, 0)$ under the isomorphism $ X_*(T) \simeq \mathbb{Z}^r$.

Let $U^{F}_{\sqrt{v}}(\hat{\mathfrak{gl}}_n)$ be the quantum group introduced in Theorem \ref{twistedhopfalgebra} with deformation parameter $\sqrt{v}$ and $\gamma_{ij} = -\frac{g(i-j)}{\sqrt{v}}$ where $g(i)$ is an $n$-th order Gauss sum as above. 
Let $V_{z^n_k}$, $1 \leqslant k \leqslant r$ be the fundamental evaluation module of $U^{F}_q(\hat{\mathfrak{gl}}_n)$ associated to $z^n_k \in \mathbb{C}^{\times}$ with basis $\{ v_i(z^n_k) \}$. Denote by $\tilde R^{\mathbf{\gamma}}(z^n_i z^{-n}_j) : V_{z^n_i} \otimes V_{z^n_j} \to V_{z^n_i} \otimes V_{z^n_j}$ the map
\[ \tilde R^{\mathbf{\alpha}}(x) = \sum_i \frac{-v+x}{1-vx} \, e_{i i} \otimes e_{i i} + \sum_{i \neq j} g(i-j)\frac{1-x}{1-vx} e_{i i} \otimes
   e_{j j} \]
\begin{equation} + \frac{1-v}{1-vx} \sum_{i > j} e_{i j} \otimes e_{j i} + x\frac{1-v}{1-vx} \sum_{i < j} e_{i j} \otimes e_{j i}, \,\,\, x = z^n_i z_j^{-n}. \label{tildeRdefined} \end{equation}
This is just the $R$-matrix associated to the quantum group  $U^{F}_q(\hat{\mathfrak{gl}}_n)$ in Section \ref{Drinfeldtwist} with a different normalization. It will satisfy the parametrized Yang-Baxter equation and a $triangularity$ property: 
\begin{equation} \label{eq:triangularity}
\tau \tilde R^{\mathbf{\gamma}}(x) \tau \tilde R^{\mathbf{\gamma}}(x^{-1}) = \left( \frac{-v+x}{1-vx} \right) \left( \frac{-v+x^{-1}}{1-vx^{-1}} \right) I = I.
\end{equation}
Note that for the last equality to make sense, we need $x \neq v^{\pm}$, which follows from the extra restriction we added to $\hat{T}_{\textrm{reg}}(\mathbb{C})$. 

There is an isomorphism of vector spaces between $\theta_{\mathbf{z}} : \mathcal{W}_{\mathbf{z}} \to V_{z^n_1} \otimes \cdots \otimes V_{z^n_r}$ that takes $W^{\chi}_{b}$, with $b = \rho + \sum_{j} i_j e_j$, to $v_{i_1}(z^n_1) \otimes \cdots \otimes v_{i_r}(z^n_r)$ (see Theorem 1 in \cite{BrubakerBuciumasBump}). This isomorphism allows us to 
relate intertwining integrals between spaces of Whittaker functionals on the metaplectic $n$-cover of $p$-adic GL$(r)$ and braidings between evaluation modules of affine quantum groups:

\begin{theorem}[Theorem 1 in \cite{BrubakerBuciumasBump}]\label{thm:BBB}
The following diagram commutes:
\[\begin{CD}\mathcal{W}_{\mathbf{z}}@>\theta_{\mathbf{z}}>>V_{z^n_1} \otimes \cdots \otimes V_{z^n_i}\otimes V_{z^n_{i+1}}\otimes\cdots\otimes V_{z^n_r}\\
@VV\bar{\Ai}_{s_i}^{\ast}V @VV {I_{V_{z^n_1}}\otimes\cdots\otimes
\tau \tilde R^{\gamma}(z_i^n z^{-n}_{i+1})\otimes\cdots\otimes I_{V_{z^n_r}}} V\\
\mathcal{W}_{s_i\mathbf{z}}@>\theta_{s_i\mathbf{z}}>>V_{z^n_1} \otimes \cdots \otimes V_{z^n_{i+1}}\otimes V_{z^n_i}\otimes\cdots\otimes V_{z^n_r}
\end{CD} \]
where $\bar{\Ai}^\ast_{s_i}$ is the map given by $ \bar{\Ai}^\ast_{s_i} (W_b^{\chi}) = W_b^{\chi} \circ  \bar{\Ai}_{s_i}$. 
\end{theorem}

Since a Whittaker functional is any functional that factors through the module
of coinvariants, $\mathcal{W}_{\mathbf{z}}$ is the dual space to the module
of coinvariants mentioned in Example~\ref{metalmtheory}.

We may now describe an instance of the representation schema related to these
evaluation modules. We will revisit the substance of Proposition~\ref{heckeybeinstanceDrinfeld}
making use of the metaplectic L-group framework. If
$\mathbf{z}= (z_1, \cdots, z_r) \in \hat{T} (\mathbb{C})$ then since
$\Lambda^{(n)}$ contains $n \Lambda$, the quantities $z_i^n$ depend only on
$\mathbf{z}' \in \hat{T}' (\mathbb{C})$, and $z_i^n / z^n_{i + 1} =
(\mathbf{z}')^{\alpha_i'}$. Let
\[ \mathcal{M} (\mathbf{z}') = V_{z_1^n} \otimes \cdots \otimes V_{z_r^n} .
\]
Define the map $A_{s_i}^{\mathbf{z}'} : \mathcal{M} (\mathbf{z}')
\longrightarrow \mathcal{M} (s_i \mathbf{z}')$ by
\begin{equation}\label{intertwiningRmatrix}
A_{s_i}^{\mathbf{z}'}  = \frac{1-vz_i^n z_{i+1}^{-n}}{1-z_i^n z_{i+1}^{-n}}  I_{V_{z^n_1}} \otimes \cdots \otimes \tau \tilde R^{\gamma}(z_i^n z_{i+1}^{-n}) \otimes \cdots \otimes I_{V_{z^n_r}}.
\end{equation}

\begin{proposition}
\label{typeametaplecticschema}
  The above setup satisfies the assumptions in Section~\ref{schemasect}.  Thus
 $\mathfrak{M}(\mathbf{z}')$ is a representation of the affine Hecke
  algebra associated with the metaplectic dual-group. As in Proposition~\ref{heckeybeinstanceDrinfeld},
  there is a commuting action of $U_{\sqrt{v}}^F(\widehat{\mathfrak{gl}}(n))$.
\end{proposition}

\begin{proof}
The first assumption is satisfied according to equation
(\ref{eq:triangularity}) with $x = z_i^n z_{i+1}^{-n}$. Note that the term
used to define $A_{s_i}^{\mathbf{z}}$ in equation \eqref{intertwiningRmatrix}
is needed for this step. The second equation follows from the fact that
$A_{s_i}^{\mathbf{z}}$ and $A_{s_j}^{\mathbf{z}}$ act on different components
of the tensor product when $|i-j|>1$. The third equation is equivalent to the
parametrized Yang-Baxter equation for $\tilde R^{\gamma}(x)$
(\cite{BrubakerBuciumasBump} Theorem~4. Since the proof of Theorems~\ref{hecketheorem}
and~\ref{hecketheoremaffine} hold for any $\mathbf{z}^{\alpha_i^\vee} = z_i
z_{i+1}^{-1}$ with each $z_i$ an arbitrary complex parameter in
$\hat{T}_{\textrm{reg}}(\mathbb{C})$, we may substitute $z_i \mapsto z_i^n$
for all $i$ and the proof goes through without change.
\end{proof}

While these last results are restricted to Cartan type $A$, we have noted that
calculations with $\bar{\Ai}_{s_i}^{\ast}$ for split groups like
(\ref{tauone}) and~(\ref{tautwo}) require only information from rank one
embeddings. Thus, in the next section, we'll be able to use the prior two
results to connect $R$-matrices to the calculation of values of the spherical
Whittaker function on metaplectic covers of split, reductive groups.

\begin{remark}
The representation in Proposition~\ref{typeametaplecticschema} (viewed as a
representation on the space of Whittaker functionals) is a variant of the
representation in Theorem~\ref{heckeybeinstanceDrinfeld} with $z_i$ replaced
by $z_i^n$ for a given $\gamma$ according to Theorem~\ref{thm:BBB}. 
We may also obtain a representation of the same affine Hecke algebra
as follows. Note that, setting $v=q^{-1}$, the representation in Example~\ref{metalmtheory}
also has a commuting action of the given metaplectic cover of $\GL(r,F)$, since it is
built from the principal series representations. So we may take
the Whittaker functionals and still have a representation of the
Hecke algebra associated with the metaplectic L-group. Now the 
operators $\theta_{\mathbf{z}}$ may be combined to produce a
single operator that is an isomorphism between this representation
and the one in Proposition~\ref{typeametaplecticschema}.
This is the connection between Hecke modules coming from intertwining
operators of representations of $p$-adic groups and those coming from
$R$-matrices of quantum groups that was mentioned in the introduction.
\end{remark}

\section{Computing special vectors in models}\label{modelsohboy}

We continue with the case of $M(\mathbf{z}) := I(\chi_{\mathbf{z}})$, a principal series on an $n$-fold metaplectic cover $\tilde{G}$ of a split, reductive
group $G$ defined over a nonarchimedean local field $F$ containing $\mu_{2n}$, the $2n$-th roots of unity. In this section, we explain
how the foundations we discussed in Section~\ref{schemasect} allow one to evaluate the image of Iwahori-fixed vectors in $M(\mathbf{z})$
under certain models for which the space of embeddings is finite dimensional. This generalizes the setting of Section~\ref{whitheck}, where this was
carried out for Whittaker and spherical functionals on (non-metaplectic) unramified principal series. 
The present section will apply, in particular, to Iwahori-fixed vectors under
metaplectic Whittaker functionals, which are treated in an extended example at
the end of the section and connected to work of
\cite{ChintaGunnellsPuskas,PuskasWhittaker,PatnaikPuskas}. As noted in the prior section, the dimension
of the space of Whittaker models for $M(\mathbf{z})$ is finite and we may choose a spanning set for this space, indexed by coset representatives for
$\Lambda / \Lambda^{(n)}$, where $\Lambda$ denotes the cocharacter lattice and $\Lambda^{(n)}$ is as in Section~\ref{whitandr}.

Given a simple reflection $s$ in $W$, the Weyl group of $G$, the calculation of Iwahori fixed vectors in any finite dimensional model hinges
on the following identity. Let $\mathcal{F}_i^{\mathbf{z}}$, $i = 1, \ldots, k$, span the finite dimensional space of embeddings of $M(\mathbf{z})$ into the model. Then with
$\Ai_{s}^{\mathbf{z}}$ the unnormalized intertwining operator from $M(\mathbf{z}) \longrightarrow M(s \mathbf{z})$ as in (\ref{intertwinerwithcs}), we must compute the ``scattering matrix'' $(b_{i,j}(\mathbf{z}))$ defined
by
\begin{equation} \mathcal{F}_i^{s \mathbf{z}} \circ \Ai_{s}^{\mathbf{z}} =
  \sum_{j=1}^k b_{i,j}^s(\mathbf{z})
  \mathcal{F}_j^{\mathbf{z}}. \label{scatmat}
\end{equation}
As $s$ is a simple reflection, this may be reduced to a rank one calculation. In Section~\ref{whitandr}, we presented formulas for the scattering matrix for Whittaker models due to Kazhdan and Patterson \cite{KazhdanPatterson} and later used in \cite{McNamaraCS} to compute spherical Whittaker functions for covers of quasi-split $p$-adic reductive groups. In this section, we provide a new
approach to performing this calculation based on the schema introduced in Section~\ref{schemasect}. It applies more broadly to any finite dimensional space of functionals
for which this scattering matrix may be computed. And it produces operators on complex vector spaces which recover the
metaplectic Demazure operators generalizing \cite{BrubakerBumpLicata} and
appearing in \cite{ChintaGunnellsPuskas, PuskasWhittaker, PatnaikPuskas}.

By a ``model'' for a representation, we mean an embedding of the representation space into a space of complex-valued functions characterized by its transformation properties under a particular subgroup. Given a finite dimensional model for principal series $M(\mathbf{z})$ on a cover of a split, reductive $p$-adic group, we will compute the image of $\phi^\circ$, the unique-up-to-constant spherical vector in $M(\mathbf{z})$, in this finite dimensional model. According to our two step induction construction of the metaplectic principal series, $\phi^\circ$ is an $i(\chi)$-valued function, but one may obtain a complex-valued function by composing with a functional in $i(\chi)^\ast$. Note that
$$ \phi^\circ = \sum_{w \in W} \phi_w $$
where $\phi_w$ is a function supported on the disjoint coset $\tilde{B} w J$,
with $J$ now denoting the embedding of the Iwahori subgroup in $\tilde{G}$,
over which the cover splits, normalized so $\phi_w(w)=1$ with representative
$w$ as in Remark~\ref{weylchoice}.
Just as in Section~\ref{whitheck}, this Iwahori subgroup $J$ is the inverse
image of the negative Borel subgroup $B^- (\mathbb{F}_q)$ under the reduction
mod $\mathfrak{p}$ map $G (\mathfrak{o}) \longrightarrow G (\mathbb{F}_q)$.
  
Given a coweight $\lambda\in\Lambda$, we then obtain
\begin{equation} \mathcal{F}_i^{\mathbf{z}}(\phi^\circ) (\pi^\lambda) = \sum_{w \in W} \mathcal{F}_i^{\mathbf{z}}(\phi_w) (\pi^\lambda). \label{sphercalc} \end{equation}

The next result shows that, under certain assumptions, each summand in the above expression may be calculated as the action of operator $T_w$ as discussed in Section~\ref{schemasect}.

\begin{proposition}\label{calc-CS} Let $G$ be a split, reductive group defined over a nonarchimedean local field $F$.  Let $M(\mathbf{z})$ denote the unramified principal series on a metaplectic cover $\tilde{G}$ of $G(F)$ induced from a character $\chi_{\mathbf{z}}$, as above. Let $\mathcal{F}_i$, $i = 1, \ldots, k$, span the finite-dimensional space of embeddings of $M(\mathbf{z})$ into the given model and let $\lambda$ be a coweight. We make the following assumption:
\begin{center}
For all $1 \leqslant i \leqslant k$, $\mathcal{F}_i^{\mathbf{z}}(\phi_1)(\pi^\lambda) = f_i^\lambda(\mathbf{z}),$ for some polynomials $f_i^\lambda \in \mathbb{C}[P^\vee]$.
\end{center}
Then for $w \in W$,
\begin{equation} \mathcal{F}_i^{\mathbf{z}}(\phi_w) (\pi^\lambda) = \left[ T_w \cdot \begin{pmatrix} f_1^\lambda(\mathbf{z}) \\ \vdots \\  f_k^\lambda(\mathbf{z}) \\ - \\ \vdots \\ - \\ f_1^\lambda(w_0 \mathbf{z}) \\ \vdots \\  f_k^\lambda(w_0 \mathbf{z}) \end{pmatrix} \right]_i \label{twinmodel} \end{equation}
where $T_w$ is an operator on $\mathfrak{M}(\mathbf{z}) = \bigoplus_{w \in W} \mathcal{M}(w \mathbf{z})$, and each $\mathcal{M}(w \mathbf{z}) \simeq \mathbb{C}^k$. We define 
$$ T_w := T_{s_{i_1}} \cdots T_{s_{i_\ell}} \quad \text{if $w = s_{i_1} \cdots s_{i_\ell}$ is reduced.} $$ 
To any simple reflection $s := s_\alpha$, the action of $T_s$ on an element of $\mathcal{M}(w\mathbf{z})$ with $\ell(sw) = 1 + \ell(w)$ is supported in $\mathcal{M}(w \mathbf{z}) \oplus \mathcal{M}(sw \mathbf{z})$. Under the isomorphisms $\mathcal{M}(w \mathbf{z}) \simeq \mathbb{C}^k$, $T_s$ acts on elements $(v_w, v_{sw}) \in \mathbb{C}^k \times \mathbb{C}^k$ by
\begin{equation} T_{s} \cdot (v_w, v_{sw}) = 
\begingroup
\renewcommand*{\arraystretch}{1.5}
\begin{pmatrix} D^{(n)}_{\alpha}(w \mathbf{z}) \cdot I_{k} & (b_{i,j}^s(sw\mathbf{z})) \\ (b_{i,j}^s(w\mathbf{z})) & D^{(n)}_{\alpha}(s w \mathbf{z}) \cdot I_{k} \end{pmatrix}
\endgroup \cdot \begin{pmatrix} v_w \\ v_{sw} \end{pmatrix}, \label{tsblockaction} \end{equation}
where $I_{k}$ is the $k \times k$ identity matrix. The operator $D^{(n)}_{\alpha}(\mathbf{z})$ is obtained from $D_i(\mathbf{z})$ in the proof of Theorem~\ref{hecketheorem} with $\alpha^\vee = \alpha_i^\vee$ by the substitution $\mathbf{z}^{\alpha^\vee} \mapsto (\mathbf{z}')^{\alpha'}$ as in (\ref{metacoroot}) and setting $v = q^{-1}$. The notation $[ \cdot ]_i$ indicates that we take the $i$-th component of the resulting vector.

Moreover, the scattering matrices $(b_{i,j}^s(\mathbf{z}))$ satisfy equation (\ref{doublesch}) with $\mathbf{z}^{\alpha^\vee} \mapsto (\mathbf{z}')^{\alpha'}$, and equation (\ref{abraidc}) for the metaplectic affine Hecke algebra $\tilde{\mathcal{H}}_v$ associated to $\hat{G}' (\mathbb{C})$. Thus, according to Theorems~\ref{hecketheorem} and~\ref{hecketheoremaffine}, $\mathfrak{M}(\mathbf{z})$ is an $\tilde{\mathcal{H}}_v$-module.
\label{specialvectorcalc}
\end{proposition}

\begin{proof} First note that our description of the operators $T_w$ suffices to determine its action on $\mathfrak{M}(\mathbf{z})$. Indeed this space is spanned by functions supported on a single summand $\mathcal{M}(v \mathbf{z})$ for each $v \in W$, so it is enough to describe the action of $T_s$ on any such $\mathcal{M}(v \mathbf{z})$.

In order to show that the operators $T_w$ are well-defined, we must prove that the $T_{s_i}$ satisfy the braid relations. This follows because the intertwining operators $\Ai_s^\mathbf{z}$ on metaplectic principal series satisfy the braid relations (see Prop.~I.2.3(a) in \cite{KazhdanPatterson}; the proof for arbitrary split, reductive groups is a similar exercise in Fubini's theorem; see Proposition 4.1 in \cite{McNamaraCS}), and hence this property is inherited by the scattering matrix $(b_{i,j})$ upon composing with the functionals $\mathcal{F}_i$. That is, they satisfy the braiding assumption (\ref{abraidc}) of Section~\ref{schemasect}. Thus the braiding of the operators $T_w$ follows according to the proof of Theorem~\ref{hecketheorem}.

To show the relation (\ref{doublesch}) with $\mathbf{z}^{\alpha_i^\vee} \mapsto (\mathbf{z}')^{\alpha_i'}$ holds for the scattering matrix $(b_{i,j}^s(\mathbf{z}))$, note the scalar appears as the composition of intertwining operators according to Proposition~4.4 of~\cite{McNamaraCS}, and hence the claim follows by linearity of the functionals $\mathcal{F}_i$. Thus, since assumptions (\ref{doublesch}) and (\ref{abraidc}) of Section~\ref{schemasect} are in force, we may apply Theorems~\ref{hecketheorem} and~\ref{hecketheoremaffine} to obtain a $\tilde{\mathcal{H}}_v$ module structure for the Hecke algebra associated to $\hat{G}' (\mathbb{C})$ as in Section~\ref{whitandr}, with $v = q^{-1}$ and $q$ the cardinality of the residue field. Here, as noted above, the action of $T_{s_i}$ uses the operator $D^{(n)}_{\alpha}(\mathbf{z})$, which is obtained from $D_i(\mathbf{z})$ in the proof of Theorem~\ref{hecketheorem} with $\alpha^\vee = \alpha_i^\vee$ by the substitution $\mathbf{z}^{\alpha^\vee} \mapsto (\mathbf{z}')^{\alpha'}.$


Finally, to prove (\ref{twinmodel}) we will show that for any $w \in W$,
\begin{equation} \mathcal{F}_i^{\mathbf{z}}(\phi_w) (\pi^\lambda) = \left[ T_w \cdot \begin{pmatrix}  \vdots \\ \mathcal{F}_i^{\mathbf{z}}(\phi_1) (\pi^\lambda) \\ \vdots \\ - \\ \vdots \\ - \\ \vdots \\ \mathcal{F}_i^{w_0 \mathbf{z}}(\phi_1) (\pi^\lambda) \\ \vdots \end{pmatrix} \right]_i, \label{needtoshowfirst} \end{equation}
and hence the result will follow from our assumption that $\mathcal{F}_i^{\mathbf{z}}(\phi_1) (\pi^\lambda) = f_i^\lambda(\mathbf{z})$. But (\ref{needtoshowfirst}) follows by induction on the length $\ell(w)$. Indeed for a simple reflection $s=s_\alpha$ and an element $w \in W$ with $\ell(sw) = \ell(w) + 1$, we have

\begin{equation} \phi_{sw}^{w \mathbf{z}} =
  \Ai_s^{s w\mathbf{z}} \phi_w^{sw \mathbf{z}}  + (c_s^{(n)}(w\mathbf{z}) - 1) \phi_w^{w \mathbf{z}}, \quad
  c_s^{(n)}(\mathbf{z})
  = \frac{1-q^{-1} \mathbf{z}^{n_\alpha \alpha^\vee}}{1 - \mathbf{z}^{n_\alpha
      \alpha^\vee}}
  \label{rankoneinter}.
\end{equation}
This is a metaplectic version of Casselman's Theorem~3.4 in
\cite{CasselmanSpherical}, which was redone in the non-metaplectic case in
Proposition 3 of \cite{BrubakerBumpLicata}. Just as in Casselman's
(non-metaplectic) proof, it may be deduced from the rank one case (i.e.,
$w=1$). The rank one metaplectic case is proved using the equality
\begin{equation}
  \Ai_s^\mathbf{z} (\phi^\circ)^{\mathbf{z}} = c_s^{(n)}(\mathbf{z})
  (\phi^\circ)^{s\mathbf{z}},
  \label{sphericalaction}
\end{equation}
a consequence of Theorems~4.2 and 12.1 of \cite{McNamaraCS}, and the fact that in rank one, $\phi^\circ = \phi_1 + \phi_s$. Thus it is enough to evaluate both sides of (\ref{sphericalaction}) at the Weyl group elements $1$ and $s$ and compute that $\Ai_s^{s\mathbf{z}} \phi_s^{s\mathbf{z}} (1) = q^{-1} \phi^\circ(1)$. This latter equality is straightforward, since $\phi_s^{\mathbf{z}}(sn) = 0$ unless $n \in J \cap N(F) = N(\mathfrak{p})$. The general case of (\ref{rankoneinter}) then follows by convolving with characteristic functions of $JwJ$ on the right in the rank one identity. The convolution calculation is identical to the non-metaplectic case since the double cosets are all in $K^\circ = G(\mathfrak{o})$ (see Remark~\ref{weylchoice}) over which the cover splits.

Applying the functional $\mathcal{F}_i$ to both sides of (\ref{rankoneinter}) and noting that the scattering matrix for the intertwining operator is given by $(b^s_{i,j}(\mathbf{z}))$, it follows that $\mathcal{F}_i^{\mathbf{z}}(\phi_w) (\pi^\lambda)$ is indeed given by (\ref{needtoshowfirst}).
\end{proof}


The proposition thus gives the computation of standard basis elements in the Iwahori fixed vectors of $M(\mathbf{z})$. We may apply it to the calculation of the spherical vector in the model according to (\ref{sphercalc}), obtaining:
\begin{equation} \mathcal{F}_i^{\mathbf{z}}(\phi^\circ) (\pi^\lambda) = \sum_{w \in W} \mathcal{F}_i^{\mathbf{z}}(\phi_w) (\pi^\lambda) = \sum_{w \in W} \left[ T_w \cdot \begin{pmatrix} f_1^\lambda(\mathbf{z}) \\ \vdots \\  f_k^\lambda(\mathbf{z}) \\ - \\ \vdots \\ - \\ f_1^\lambda(w_0 \mathbf{z}) \\ \vdots \\  f_k^\lambda(w_0 \mathbf{z}) \end{pmatrix} \right]_i. \label{genericspherical} \end{equation}

\subsection*{Example: Metaplectic Whittaker functions}

 The assumptions of the above proposition hold for metaplectic Whittaker functionals. Indeed, the space of such functionals is finite dimensional, and we may choose a basis of metaplectic Whittaker functionals as in Section~\ref{whitandr}, indexed by coset representatives $\nu$ for $\Lambda / \Lambda^{(n)}$. 
 
 \begin{lemma}
 To any coset representative $\nu$ in $\Lambda / \Lambda^{(n)}$, for any dominant coweight $\lambda$,
$$\mathcal{W}_{\nu}(\phi_1)(\pi^{-\lambda}) = \begin{cases} 0 & \textrm{if $\nu \not\equiv \lambda \pmod{\Lambda^{(n)}}$} \\ \mathbf{z}^{-\lambda} & \textrm{if $\nu \equiv \lambda \pmod{\Lambda^{(n)}}$}. \end{cases} $$ 
\label{supportoncosetslemma}
\end{lemma}

\begin{proof}
Recall that $i(\chi)$ was the induced representation on $\tilde{T}$ from a
character $\chi$ on the maximal abelian subgroup $H$. The space
$I(\chi)^{K^\circ} \simeq i(\chi)^{K^\circ \cap \tilde{T}}$ under the map
$\phi^\circ \mapsto \phi^\circ(1) = \phi_1(1)$, where the equality follows
from the Iwahori-Bruhat decomposition. Set $v_0 = \phi_1(1)$. Then $\{
\pi_\chi(a) v_0 \}$ with $a \in \tilde{T} / H$ is a basis for the
representation $i(\chi)$. These may be indexed by elements $\mu \in \Lambda /
\Lambda^{(n)}$ with $a = \pi^\mu$.  The Whittaker functionals $W_{\nu}$ are
obtained by composing the $i(\chi)$-valued Whittaker functional with a
functional on $i(\chi)$. These functionals may be indexed by elements $\nu$ in
the dual basis for $i(\chi)^\ast$, likewise indexed by $\Lambda /
\Lambda^{(n)}$.  Now use Lemma~6.3 of McNamara \cite{McNamaraCS}, which
connects the right translation action by $\tilde{T}$ on Whittaker functions
(resulting in evaluation at $t^\lambda$ with $\lambda \in \Lambda$) to the
basis of $i(\chi)$ made from translates of $\phi^\circ(1) = \phi_1(1)$. This
is accomplished using the Iwahori factorization. Note that our Whittaker
functional is defined without the use of the long element $w_0$, so applies to
the standard basis element $\phi_1$ rather than McNamara's $\phi_{w_0}$ in his
Lemma 6.3.
\end{proof}

The lemma shows that precisely one component in each of the $|W|$ blocks for the vector $(\mathcal{W}_{\nu}^{w \mathbf{z}}(\phi_1) (\pi^{-\lambda}))_{\nu,w}$ is non-zero. Thus Proposition~\ref{specialvectorcalc} applies to all dominant coweights $\lambda$. (In fact, it applies to all coweights, but the Whittaker functional is easily seen to vanish off of the dominant cone.) The resulting formula for the spherical Whittaker function from (\ref{genericspherical}) is analogous to the theorem and corollary in Section 5.4 of \cite{PatnaikPuskas}. 


We end by explaining that the operators obtained by the action of $T_w$ on $(\mathcal{W}_{\nu}^{w \mathbf{z}}(\phi_1) (\pi^{-\lambda}))_{i,w}$ are closely related to the metaplectic Demazure operators of \cite{ChintaGunnellsPuskas}, using the metaplectic action of the Weyl group on the coweight lattice as given by Chinta and Gunnells. Indeed, these operators arise in the calculation of a particular Whittaker functional $\mathcal{W}_\Sigma := \sum_{\nu \in \Lambda / \Lambda^{(n)}} W_\nu$. It is this metaplectic Whittaker functional, evaluated on the spherical vector, that has been repeatedly explored in earlier papers.

Recall that to a simple reflection $s := s_\alpha$, given any representative coweight $\mu$ in the quotient lattice $\Lambda / \Lambda^{(n)}$,
one can define the metaplectic Chinta-Gunnells action $s \cdot f$ for Laurent polynomials
$f \in \mathbb{C}[\Lambda]$ supported on coweights in $\mu + \Lambda^{(n)}$. It may extended to the fraction field of $\mathbb{C}[\Lambda]$ (see for example Section~3 of \cite{ChintaGunnells} or Section~2 of \cite{ChintaGunnellsPuskas}).
Indeed, according to Equation (7) in \cite{ChintaGunnellsPuskas},
\begin{multline} c^{(n)}_{s_i}(\mathbf{z}) s_i \cdot f := f^{s_i} \left[ \mathbf{z}^{-\textrm{rem}_{n_\alpha} \left(\frac{-B(\alpha^\vee, \mu)}{Q(\alpha^\vee)} \right) \alpha^\vee} (1-q^{-1}) (1-\mathbf{z}^{n_\alpha \alpha^\vee})^{-1}  \right. \\ \left. - q^{-1} g(Q(\alpha^\vee)-B(\alpha^\vee, \mu)) \mathbf{z}^{(1-n_\alpha) \alpha^\vee} \right]. \label{cgaction} \end{multline}
Recall that $B(\alpha^\vee, \mu)$ is the bilinear form accompanying the construction of the metaplectic cover, with corresponding quadratic form $Q(\alpha^\vee) := B(\alpha^\vee, \alpha^\vee) / 2,$ and $Q(\alpha^\vee)$ divides $B(\alpha^\vee, \mu)$ according to \cite{McNamaraPrincipal}, Theorem 11.1. Further, $\textrm{rem}_{n(\alpha)}$ denotes the remainder mod $n_\alpha$, taking values in $[0, n-1]$, upon division of the integral quantities in the argument. Moreover, $g(i)$ is the $n$-th order Gauss sum described in Section~\ref{whitandr}. Note that the action is well-defined as it is independent of coweight representative $\mu \in \Lambda / \Lambda^{(n)}$. It can be extended to all polynomials $f$ by additivity. The function $c^{(n)}_s(\mathbf{z})$ is as in (\ref{rankoneinter}).

Then the metaplectic Demazure operator $\mathcal{T}_i$ corresponding to the simple reflection $\alpha_i$ is defined in Equation (11) of \cite{ChintaGunnellsPuskas} by
\begin{equation} \mathcal{T}_i (f) := D_i^{(n)}(\mathbf{z}) f - \mathbf{z}^{n_{\alpha_i} \alpha_i^\vee} c^{(n)}_{s_i}(\mathbf{z}) s_i \cdot f. \label{metdemop} \end{equation}
Then $\mathcal{T}_w$ is defined as a the product $\mathcal{T}_{i_1} \cdots \mathcal{T}_{i_\ell}$ if $w = s_{i_1} \cdots s_{i_\ell}$ is reduced, as usual.

\begin{theorem}\label{met-dz}
The metaplectic Demazure operator $\mathcal{T}_w$ defined from (\ref{metdemop}) agrees with the action of the $T_w$ defined according to (\ref{tsblockaction}) for the spherical Whittaker function with Whittaker functional $W_\Sigma$. 
\end{theorem}

\begin{proof}
It suffices to check this for simple reflections $s_i$. We will show the two actions agree upon setting $v = q^{-1}$. 

For a simple reflection, the action of $T_{s_i}$ is by the block matrix in (\ref{tsblockaction}). The blocks containing the main diagonal act by the scalar $D_i^{(n)}(\mathbf{z})$. The terms in $f$ supported on $\mathbf{z}^\nu + \Lambda^{(n)}$ will be detected by the functional $W_\nu$, and thus appear in the $\nu$-th component of the column vector in (\ref{genericspherical}) according to Lemma~\ref{supportoncosetslemma}. In $W_\Sigma$, we sum over all $W_\nu$ and thus the total contribution to $W_\Sigma$ from diagonal blocks is just $D_i^{(n)}(\mathbf{z}) f$ as desired. Thus it suffices to show that that off-diagonal blocks in the action of $T_{s_i}$ match the second term in $\mathcal{T}_i$ above using the Chinta-Gunnells $W$-action.

For brevity, let $s=s_i$ and $\alpha^\vee = \alpha_i^\vee$. The off-diagonal blocks in $T_{s}$ act on the component $M(s\mathbf{z})$ in $\mathfrak{M}(\mathbf{z})$ by the scattering matrix $(b_{i,j}^s(s\mathbf{z}))$ according to (\ref{tsblockaction}). Each row of $(b_{i,j}^s)$ has two non-zero coefficients where the indices $(i,j)$ correspond to pairs of coset representatives $(\nu, \mu)$ in $\Lambda / \Lambda^{(n)}$ with either $\nu \equiv \mu$ (mod $\Lambda^{(n)}$) or $\nu \equiv s(\mu) + \alpha^\vee$ (mod $\Lambda^{(n)}$). In the notation of Section~\ref{whitandr}, the coefficients $b_{\mu, \mu}$ and $b_{s(\mu) + \alpha^\vee, \mu}$ are equal to $c_s^{(n)}(\mathbf{z}) \tau^1$ and $c_s^{(n)}(\mathbf{z}) \tau^2$, respectively, with $\tau^1$ and $\tau^2$ as in (\ref{tauone}) and (\ref{tautwo}). In particular, the $\tau$'s are multiplied by $c_s^{(n)}(\mathbf{z})$ because Proposition~\ref{calc-CS} is carried out with unnormalized intertwining operators appearing in (\ref{scatmat}). Thus the off-diagonal block acts on $M(s \mathbf{z})$ in $\mathfrak{M}(\mathbf{z})$, producing 
$$ c_{s}^{(n)}(s \mathbf{z}) \tau^1_{\mu, \mu}(s \mathbf{z}) \cdot (s \mathbf{z})^\mu + c_{s}^{(n)}(s \mathbf{z}) \tau^2_{s(\mu) + \alpha^\vee, \mu} (s\mathbf{z}) \cdot (s \mathbf{z})^{\mu} $$ 
in the model $W_\mu$. Thus in $W_\Sigma$, we take the sum of these contributions. 
Substituting the formulas for $\tau^1$ and $\tau^2$ and comparing to $- \mathbf{z}^{n_{\alpha} \alpha^\vee} c^{(n)}_{s}(\mathbf{z}) s \cdot f$ appearing in the metaplectic Demazure operator with $c^{(n)}_{s}(\mathbf{z}) s \cdot f$ as in (\ref{cgaction}) gives the result.
Note in particular that 
$$ n_\alpha \left\lceil \frac{B(\alpha^\vee, \mu)}{n_\alpha Q(\alpha^\vee)} \right\rceil -  \frac{B(\alpha^\vee, \mu)}{Q(\alpha^\vee)} = \textrm{rem}_{n_\alpha} \!\! \left(\frac{-B(\alpha^\vee, \mu)}{Q(\alpha^\vee)} \right), $$
and that the Gauss sums differ from each other as one involves the conjugate multiplicative character to the other. This difference amounts to a choice of embedding of the roots of unity in our local field $F$ into the complex numbers, so we may alter this choice in order to achieve an exact match.
\end{proof}

To conclude, let us explain how metaplectic Demazure operators
for arbitrary Cartan types may be written in terms of
$R$-matrices. More precisely, the operator $\mathcal{T}_{i}$ for a simple reflection $s = s_{\alpha_i}$ 
defined in (\ref{metdemop}) is expressed as the difference of a scalar times the function $f$ and a term $ \mathbf{z}^{n_{\alpha_i} \alpha_i^\vee} c^{(n)}_{s_i}(\mathbf{z}) s_i \cdot f$ which uses the Chinta-Gunnells metaplectic action. It is this latter term that may be rewritten in terms of $R$-matrices.

Indeed, the previous proof demonstrates that this term is expressible in terms of the scattering matrix for intertwining operators acting on the Whittaker models
of principal series. By Theorem~\ref{thm:BBB} with $r=2$, this action is in turn expressible using $R$-matrices.
The key point is that while the results of Section 5 are for type A,
the Demazure operator for a simple reflection requires only a rank one
computation. Thus we may restrict our attention to an embedded rank one
subgroup and apply Theorem~\ref{thm:BBB} with $r=2$ to the tensor of two evaluation modules $V_{\mathbf{z}^{n_\alpha \alpha}} \otimes V_{s \mathbf{z}^{n_\alpha \alpha}}$ as defined in Section~\ref{whitandr}.  This latter term may be written in terms of the $R$-matrices $\tilde{R}^\gamma(\mathbf{z}^{n_\alpha \alpha})$ as in (\ref{tildeRdefined}). Note that these $R$-matrices come from evaluation modules for $U_{\sqrt{v}}(\hat{\mathfrak{gl}}(n_\alpha))$.

\begin{appendix}

\section{A SAGE program to check Theorem~\ref{hecketheorem}}\label{appendix}

Below we provide a program written in SAGE that checks 
Theorem~\ref{hecketheorem}. It is sufficient to check
(\ref{quadratic}) and (\ref{firstbraid}) for the rank two
Cartan types $A_2$, $C_2$ and $G_2$.

  \begin{verbatim}
"""
The function check_relations() defined in this file
confirms (11) and (12) in Theorem 1.1. Tested with sage 7.4.
"""

# Edit the following line to select the Cartan Type "A2", "C2" or "G2".
CT = "G2"

P.<z1,z2,z3,v>=PolynomialRing(QQ)
W = WeylGroup(CT, prefix="s")
alpha = RootSystem(CT).ambient_space().simple_roots()

A={}
for i in [1,2]:
    for w in W:
        # indeterminates representing intertwining operators
        A[(w,i)]='a'+"%s_"%i+"".join(["%s"%j for j in w.reduced_word()])
P=PolynomialRing(QQ,A.values()+['z1','z2','z3','v'])
P.inject_variables()

# The following substitution is justified by (7) in the hypotheses of Theorem 1.1.
if CT == "A2":
    a2_121=a1_121*a2_21*a1_1/(a1_12*a2_2)
elif CT == "C2":
    a2_2121=a2_2*a1_12*a2_212*a1_2121/(a1_1*a2_21*a1_121)
elif CT == "G2":
    a2_212121=(a1_1212*a2_2*a2_21212*a1_12*a1_212121*a2_212)/ \
               (a2_2121*a2_21*a1_1*a1_121*a1_12121)

for i in [1,2]:
    for w in W:
        A[(w,i)] = eval(A[(w,i)])

if CT in ["A2","G2"]:
    def C(r):
        za = z1^(r[0])*z2^(r[1])*z3^(r[2])
        return (1-v*za)/(1-za)
elif CT == "C2":
    def C(r):
        za = z1^(r[0])*z2^(r[1])
        return (1-v*za)/(1-za)

def t(i, w, y):
    s = W.simple_reflection(i)
    if y==w:
        return C(alpha[i].weyl_action(w.inverse()))-1
    elif y==s*w:
        if y.bruhat_le(w):
            return A[(w,i)]
        else:
            # Constant from Assumptions: see (6)
            return C(alpha[i].weyl_action(y.inverse()))* \
                C(alpha[i].weyl_action(w.inverse()))/A[(y,i)]
    else:
        return 0

[T1,T2] = [Matrix([[t(i,w,y) for w in W] for y in W]) for i in [1,2]]

def check_relations():
    """
    Confirms the braid and quadratic relations. Usage:
    sage: check_relations()
    [True, True, True]
    """
    if CT == "A2":
        rel3 = (T1*T2*T1 == T2*T1*T2)
    elif CT == "C2":
        rel3 = (T1*T2*T1*T2 == T2*T1*T2*T1)
    elif CT == "G2":
        rel3 = (T1*T2*T1*T2*T1*T2 == T2*T1*T2*T1*T2*T1)
    return [T1^2 == (v-1)*T1 + v, T2^2 == (v-1)*T2 + v, rel3]

\end{verbatim}

\end{appendix}

\bibliographystyle{hplain}
\bibliography{sjduality.bib}

\end{document}